\documentclass[12pt, english]{article}
\usepackage[utf8]{inputenc}
\usepackage[T1]{fontenc}
\usepackage{babel}
\usepackage{float}
\usepackage{amsmath}
\usepackage{amsthm}
\usepackage{graphicx}
\usepackage[unicode=true,pdfusetitle,
 bookmarks=true,bookmarksnumbered=false,bookmarksopen=false,
 breaklinks=false,pdfborder={0 0 1},backref=page,colorlinks=false]
 {hyperref}

\makeatletter
\numberwithin{figure}{section}
\theoremstyle{plain}
\newtheorem{thm}{\protect\theoremname}
\newenvironment{lyxlist}[1]
	{\begin{list}{}
		{\settowidth{\labelwidth}{#1}
		 \setlength{\leftmargin}{\labelwidth}
		 \addtolength{\leftmargin}{\labelsep}
		 }}
	{\end{list}}
\theoremstyle{plain}
\newtheorem{prop}[thm]{\protect\propositionname}
\theoremstyle{plain}
\newtheorem{cor}[thm]{\protect\corollaryname}
\theoremstyle{definition}
\newtheorem{defn}[thm]{\protect\definitionname}
\theoremstyle{remark}
\newtheorem{rem}[thm]{\protect\remarkname}
\theoremstyle{plain}
\newtheorem{lem}[thm]{\protect\lemmaname}

\usepackage{physics,courier,afterpage,enumitem,amsthm}
\IfFileExists{mtpro2.sty}{\usepackage{libertine}\usepackage[subscriptcorrection,mtphrd,mtpcal,amssymbols]{mtpro2}}{%
  \providecommand*{\comp}{\circ}%
  \usepackage{lmodern}
  \usepackage{amsmath}
  \usepackage{amssymb}
}

\usepackage{graphicx,caption,tikz,xparse,float}
\usepackage[left=3.8cm,right=3.8cm,bottom=3.5cm,top=3.5cm]{geometry}

\usepackage{tikz-cd}
\theoremstyle{definition}

\theoremstyle{theorem}

\newtheorem{innercustomthm}{Theorem}
\newenvironment{customthm}[1]
  {\renewcommand\theinnercustomthm{#1}\innercustomthm}
  {\endinnercustomthm}
\newtheorem{innercustomlemma}{Lemma}
\newenvironment{customlem}[1]{\renewcommand\theinnercustomlemma{#1}\innercustomlemma}{\endinnercustomlemma}
\newtheorem{innercustomcor}{Corollary}
\newenvironment{customcor}[1]
  {\renewcommand\theinnercustomcor{#1}\innercustomcor}
  {\endinnercustomcor}
\newtheorem{innercustomprop}{Proposition}
\newenvironment{customprop}[1]
  {\renewcommand\theinnercustomprop{#1}\innercustomprop}
  {\endinnercustomprop}

\usepackage{titlesec}
\titleformat{\subsubsection}[runin]
               {\normalfont\bfseries}
               {\thesubsubsection.}{.5em}{}[.---]
             \titlespacing{\subsubsection}
               {0pt}{1.5ex plus .1ex minus .2ex}{0pt}

\newcommand{\ens}[1]{\left\lbrace #1 \right\rbrace}

\newcommand{\del}{\partial}
\newcommand{\RR}{\mathbb{R}}
\newcommand{\Ac}{\mathcal{A}}

\DeclareMathOperator{\Spec}{Spec}
\newcommand{\iprod}{%
    \mathbin{\scalebox{1.2}{$\lrcorner$}}%
}
\DeclareMathOperator{\im}{im}

\DeclareMathOperator{\ind}{ind}
\DeclareMathOperator{\Crit}{Crit}
\newcommand{\Db}{D_\text{b}}
\newcommand{\Dc}{D_\text{c}}

\DeclareMathOperator{\Hf}{HF}
\DeclareMathOperator{\CF}{CF}
\DeclareMathOperator{\SH}{SH}
\DeclareMathOperator{\CZ}{CZ}
\DeclareMathOperator{\HH}{H}
\DeclareMathOperator{\CO}{C}
\DeclareMathOperator{\Sk}{Sk}
\newcommand{\id}{\mathrm{id}}
\newcommand{\Phib}{\Phi^{\text{b}}}
\newcommand{\Phic}{\Phi^{\text{c}}}
\DeclareMathOperator{\Ham}{Ham}
\DeclareMathOperator{\Symp}{Symp}
\DeclareMathOperator{\diam}{diam}

\makeatother

\providecommand{\corollaryname}{Corollary}
\providecommand{\definitionname}{Definition}
\providecommand{\lemmaname}{Lemma}
\providecommand{\propositionname}{Proposition}
\providecommand{\remarkname}{Remark}
\providecommand{\theoremname}{Theorem}

\begin{document}
\begin{center} 	
\textbf{THE SPECTRAL DIAMETER OF A LIOUVILLE DOMAIN}\\ Pierre-Alexandre Mailhot
\end{center}

\begin{abstract}
	\noindent 
The group of compactly supported Hamiltonian diffeomorphisms of a symplectic manifold is endowed with a natural bi-invariant distance, due to Viterbo, Schwarz, Oh, Frauenfelder and Schlenk, coming from spectral invariants in Hamiltonian Floer homology. This distance has found numerous applications in symplectic topology. However, its diameter is still unknown in general. In fact, for closed symplectic manifolds there is no unifying criterion for the diameter to be infinite.
In this paper, we prove that for any Liouville domain this diameter is infinite if and only if its symplectic cohomology does not vanish. This generalizes a result of Monzner-Vichery-Zapolsky and has applications in the setting of closed symplectic manifolds.
\end{abstract}

\tableofcontents{}

%%%%%%%%%%%%%%%%%%%%%%%%%%%%%%%%%
%%%%%%%%%%%%%%%%%%%%%%%%%%%%%%%%%
%%%%%%%%%%%%%%%%%%%%%%%%%%%%%%%%%
% INTRODUCTION
%%%%%%%%%%%%%%%%%%%%%%%%%%%%%%%%%
%%%%%%%%%%%%%%%%%%%%%%%%%%%%%%%%%
%%%%%%%%%%%%%%%%%%%%%%%%%%%%%%%%%
\section{Introduction and results}

Liouville domains are a special kind of compact symplectic manifolds
with boundary. They are characterized by their exact symplectic form $\omega = \dd \lambda$ and the fact that their boundary is of contact type. Given that they do not close up, they are quite easy to construct and allow us to
study under a common theoretical framework many important classes
of symplectic manifolds. Examples of such manifolds include cotangent
disk bundles over closed manifolds, complements of Donaldson divisors
\cite{Gi17}, preimages of some intervals under exhausting functions
of Stein manifolds \cite{CiEl12}, positive regions of convex hypersurfaces
in contact manifolds \cite{Gi91} and total spaces of Lefschetz fibrations.

%%%%%%%%%%%%%%%%%%%%%%%%%%%%%%%%%
% Motivation of SH
%%%%%%%%%%%%%%%%%%%%%%%%%%%%%%%%%
A key invariant of a Liouville domain $D$ is its symplectic cohomology
$\SH^{\ast}(D)$. It was first defined by Cielieback, Floer and Hofer
\cite{FlHo94,CiFlHo95} and later developed by Viterbo \cite{Vi99}.
Symplectic cohomology allows one to study the behavior of periodic
Reeb orbits on the boundary of $D$. It is defined in terms of the
Floer cohomology groups of a specific class of Hamiltonian functions
on the completion $\hat{D}$ of $D$ which results from the gluing
of the cylinder $[1,\infty)\times\partial D$ to $\partial D$. 

%%%%%%%%%%%%%%%%%%%%%%%%%%%%%%%%%
% Introduction of c(\alpha, H)
%%%%%%%%%%%%%%%%%%%%%%%%%%%%%%%%%
The primary goal of this paper is to relate symplectic cohomology and spectral invariants,
an important characteristic in Hamiltonian dynamics. When defined
on a symplectic manifold $(M,\omega)$, spectral invariants associate
to any pair $(\alpha,H)\in\HH^{\ast}(M)\times C_{c}^{\infty}(S^{1}\times M)$
a real number $c(\alpha,H)$, that belongs to the spectrum of the
action functional associated to $H$\footnote{at least if the Hamiltonian satisfies certain technical conditions.}.
Spectral invariants were first defined in $\RR^{2n}$ from the point
of view of generating functions by Viterbo in \cite{Vi92}. They were
then constructed on closed symplectically aspherical manifolds by
Schwarz in \cite{Sc00} and general closed symplectic manifolds by
Oh in \cite{Oh05} (see also \cite{Us13}).

In \cite{FrSc07}, Frauenfelder and Schlenk construct spectral invariants
on Liouville domains. 
These spectral invariants are homotopy invariant in the Hamiltonian
term in the following sense. If two compactly supported Hamiltonians
$H$ and $F$ generate the same time-one map, $\varphi_{H}=\varphi_{F}$,
then $c(\alpha,H)=c(\alpha,F)$. Thus $c(\alpha,\bullet)$ descends
to the group of compactly supported Hamiltonian diffeomorphisms $\Ham_{c}(D)$.
This allows one to define a bi-invariant norm $\gamma$ on $\Ham_{c}(D)$, called
the spectral norm, by 
\begin{align*}
\gamma(\varphi)=c(1,\varphi)+c(1,\varphi^{-1}) & .
\end{align*}

%%%%%%%%%%%%%%%%%%%%%%%%%%%%%%%%%
% Spectral diameter 
%%%%%%%%%%%%%%%%%%%%%%%%%%%%%%%%%
One key feature of the spectral norm $\gamma$ is the fact that it
acts as a lower bound to the celebrated Hofer norm introduced by Hofer
in \cite{Ho90} (see the article of Lalonde and McDuff \cite{LaMc95}
and the book of Polterovich \cite{Po01} for further developments
in the subject). It is thus natural to ask whether the spectral diameter
\begin{align*}
\diam_{\gamma}(M)=\sup\{\gamma(\varphi)\mid\varphi\in\Ham_{c}(M)\}
\end{align*}
is finite or not. In particular, if $\diam_{\gamma}(M)=+\infty$ then
the Hofer norm is assured to be unbounded. Further links between the spectral norm and Hofer geometry are discussed in Section \ref{sec:Hofer geometry}.

%%%%%%%%%%%%%%%%%%%%%%%%%%%%%%%%%
%%%%%%%%%%%%%%%%%%%%%%%%%%%%%%%%%
% Main results
%%%%%%%%%%%%%%%%%%%%%%%%%%%%%%%%%
%%%%%%%%%%%%%%%%%%%%%%%%%%%%%%%%%
\subsection{Main results}
In this article, we find a characterization of the finiteness of $\diam_{\gamma}(D)$ in the case of a Liouville domain $(D,\dd\lambda)$ in terms of its symplectic cohomology.

%%%%%%%%%%%%%%%%%%%%%%%%%%%%%%%%%
% If and only if theorem
%%%%%%%%%%%%%%%%%%%%%%%%%%%%%%%%%

Our main technical result shows that if $\SH^{\ast}(D)\neq0$ then $c(1,H)$ can be made arbitrarily large. This, combined with the converse implication which was proved by Benedetti and Kang \cite{BeKa20}, implies 

\begin{customthm}{A1}\label{thm:finite_diameter}
 	Let $(D,\lambda)$ be a Liouville domain. Then $\diam_\gamma(D)=+\infty$ if and only if $\SH^\ast(D)\neq 0$.
\end{customthm}

%%%%%%%%%%%%%%%%%%%%%%%%%%%%%%%%%
% c(1,H) >= 0
%%%%%%%%%%%%%%%%%%%%%%%%%%%%%%%%%
\noindent As an intermediate step to proving Theorem \ref{thm:finite_diameter},
we show the following auxiliary result.

\begin{customlem}{B}\label{lem:C}\label{lem:c geq 0}
	Let $H$ be a compactly supported Hamiltonian on a Liouville domain $(D,\lambda)$. Then, $$c(1, H)\geq 0.$$ 
\end{customlem}

\noindent Lemma \ref{lem:c geq 0} is a cohomological adaptation
of \cite[Lemma 4.1]{GaTa20}. 

%%%%%%%%%%%%%%%%%%%%%%%%%%%%%%%%%
% Sharper
%%%%%%%%%%%%%%%%%%%%%%%%%%%%%%%%%
In fact, when the symplectic cohomology of a Liouville domain is non-vanishing,
the implication of Theorem \ref{thm:finite_diameter} follows from
a sharper result. Denote by $d_{\gamma}(\varphi,\psi)=\gamma(\varphi\comp\psi^{-1})$
the spectral distance on $\Ham_{c}(D)$ and by $d_{\text{{st}}}$
the standard Euclidean distance on $\RR$.

\begin{customthm}{A2}\label{thm:B2}
 	Let $(D,\lambda)$ be a Liouville domain such that $\SH^\ast(D)\neq 0$. Then there exists an isometric group embedding $(\RR,d_{\text{st}})\to (\Ham_c(D),d_\gamma)$.
\end{customthm}The proof of Theorem \ref{thm:B2} uses an explicit construction of
an isometric group embedding. This construction is a generalization
of the procedure used by Monzner and Vichery and Zapolsky to prove
Theorem \ref{thm:Embedding_R} below. The construction of the aforementioned
embedding relies primarily on the computation of spectral invariants
of Hamiltonians which are constant on the skeleton of $D$, a special
subset of Liouville domains which we define in Section \ref{subsec:Completion-of-Liouville}.

\begin{customlem}{C}\label{lem:skeleton}
	Suppose $(D,\lambda)$ is a Liouville domain such that $\SH^\ast(D)\neq 0$. Let $H$ be a compactly supported autonomous Hamiltonian on $D$ such that $$ H\big|_{\Sk(D)}=-A\quad  \text{and}\quad  -A\leq H\big|_{D}\leq 0$$ for a constant $A>0$. Then
$$
c(1,H) = A.
$$
\end{customlem}

%%%%%%%%%%%%%%%%%%%%%%%%%%%%%%%%%
%%%%%%%%%%%%%%%%%%%%%%%%%%%%%%%%%
% Work of BK
%%%%%%%%%%%%%%%%%%%%%%%%%%%%%%%%%
%%%%%%%%%%%%%%%%%%%%%%%%%%%%%%%%%
\subsection{What is already known for Liouville domains}
\noindent Following the work of Benedetti and Kang \cite{BeKa20}, it is known that the spectral diameter of a Liouville domain $D$ is bounded if its symplectic cohomology vanishes. This result was achieved using a special capacity derived from the filtered symplectic cohomology of $D$. To better understand how this is done, let us give an overview of the construction of $\SH^{*}(D)$ following
\cite{Vi99}.   

%%%%%%%%%%%%%%%%%%%%%%%%%%%%%%%%%
% Overview of SH
%%%%%%%%%%%%%%%%%%%%%%%%%%%%%%%%%
Focusing our attention to the class of Hamiltonians,
called admissible, which are affine\footnote{See Definition \ref{def:Affine_hamiltonian} for the precise conditions.}
in the radial coordinate on the cylindrical part of $\hat{D}$, filtered
cohomology groups $\Hf_{(a,b)}^{\ast}(H)$ are well defined and only
depend on the slope of the Hamiltonian on $[1,\infty)\times\partial D$.
Taking an increasing sequence of admissible Hamiltonians $\ens{H_{i}}_{i}$
with corresponding slopes $\ens{\tau_{i}}_{i}$ satisfying $\tau_{i}\to+\infty$,
one can define the filtered symplectic cohomology $\SH_{(a,b)}^{\ast}(D)$
of $D$ as 
\begin{align*}
\SH_{(a,b)}^{\ast}(D)=\lim_{\substack{\longrightarrow\\
H_{i}
}
}\Hf_{(a,b)}^{\ast}(H_{i}).
\end{align*}
It follows from the above definition that for $a\leq a'$ and $b\leq b'$
there is a natural map $\iota_{a,a'}^{b,b'}:\SH_{(a,b)}^{\ast}(D)\to\SH_{(a',b')}^{\ast}(D)$.
Moreover, the full symplectic cohomology $\SH^{\ast}(D)=\SH_{(-\infty,\infty)}^{\ast}(D)$
comes with a natural map
\begin{align*}
v^{\ast}:\HH^{\ast}(D)\to\SH^{\ast}(D).
\end{align*}
called the Viterbo map. The failure of $v^{\ast}$ to be an isomorphism signals the presence
of Reeb orbits on the boundary of $D$. Thus, $\SH^{\ast}(D)$ is
a useful tool to study the Weinstein conjecture \cite{We79} which
claims that on any compact contact manifold, the Reeb vector field
should admit at least one periodic orbit. For instance, in \cite{Vi99},
Viterbo proves the Weinstein conjecture for the boundary of subcritical
Stein manifolds. Note that the symplectic cohomology approach to the
Weinstein conjecture has limitations as it relies on finding a Liouville
filling of the given contact manifold. 

We can extend any compactly supported Hamiltonian $H\in C_{c}^{\infty}(S^{1}\times D)$
to an admissible Hamiltonian with small slope $H^{\varepsilon}$ and define its Floer cohomology as $\Hf^{\ast}(H)=\Hf^{\ast}(H^{\varepsilon})$. A key property of Floer cohomology on Liouville domains is that if an admissible Hamiltonian $F$ has a slope close enough to zero, then we have an isomorphism $\Phi_{F}:\HH^{\ast}(D)\to\Hf^{\ast}(F)$. Thus, the Floer cohomology of compactly supported Hamiltonians on
$D$ is well defined. 

Let $H$ be a compactly supported Hamiltonian. Followig \cite{FrSc07}, the spectral invariant associated to $(\alpha,H)\in\HH^{\ast}(D)\times C_{c}^{\infty}(S^{1}\times D)$
corresponds to the real number 
\begin{align*}
c(\alpha,H)=\inf\ens{c\in\RR\mid\Phi_H(\alpha)\in\im\iota^{<c}}
\end{align*}
where 
\begin{align*}
\iota^{<c} & =\iota_{-\infty,-\infty}^{c,+\infty}:\Hf_{(-\infty,c)}^{\ast}(H)\to\Hf^{\ast}(H)
\end{align*}
is the map induced by natural inclusion of subcomplexes.

%%%%%%%%%%%%%%%%%%%%%%%%%%%%%%%%%
% Benedetti and Kang theorem
%%%%%%%%%%%%%%%%%%%%%%%%%%%%%%%%%
 Now, define the SH-capacity of $D$ as 
\begin{align*}
c_{\SH} & (D)=\inf\ens{c>0\mid\iota_{-\infty,-\infty}^{\varepsilon,c}=0}\in(0,\infty],
\end{align*}
where, for $\epsilon>0$ sufficiently small, $\iota_{-\infty,-\infty}^{\varepsilon,c}:\SH_{(-\infty,\varepsilon)}^{\ast}(D)\to\SH_{(-\infty,c)}^{\ast}(D)$.
It is known that $c_{\text{SH}}(D)$ is finite if and only if $\SH^{\ast}(D)$
vanishes. Using this, Benedetti and Kang prove the following upper
bound on spectral invariants of compactly supported Hamiltonians with respect to the unit. 

\begin{thm}[\cite{BeKa20}]\label{thm:B}
	Let $(D,\dd\lambda)$ be a Liouville domain with $\SH^\ast(D)=0$. Then,
	$$\sup\lbrace c(1,H)\rbrace \leq c_{\SH}(D)<+\infty,$$
where the supremum is taken over all compactly supported Hamiltonians in $D$.
\end{thm}

\noindent In particular, by definition of the spectral norm, if $\SH^{\ast}(D)=0$, then for any compactly supported Hamiltonian $H$ generating $\varphi_{H}\in\Ham_{c}(D)$, we have
\begin{align*}
\gamma(\varphi) & =c(1,\varphi)+c(1,\varphi^{-1})\leq2c_{\SH}(D)<+\infty.
\end{align*}
Therefore, Theorem \ref{thm:B} provides the \emph{only if} part of Theorem \ref{thm:finite_diameter}.

%%%%%%%%%%%%%%%%%%%%%%%%%%%%%%%%%
% Contains exact Lagrangian => SH \neq 0
%%%%%%%%%%%%%%%%%%%%%%%%%%%%%%%%%
On the other hand, symplectic cohomology is known to be non-zero in many cases \cite[ Section 5]{Se08}.
Since we will be using $\mathbb{Z}_{2}$ coefficients throughout this article, one case of particular interest to us is the following.
\begin{prop}[\cite{Vi99}]
 \label{prop:Viterbo_Lagrangian}Suppose $D$ contains a closed exact
Lagrangian submanifold $L$. Then, $\SH^{\ast}(D)\neq0$.
\end{prop}

\noindent This result of Viterbo can be used, in conjunction with Theorem \ref{thm:finite_diameter}, to prove that the spectral
diameter is infinite for quite general classes of Liouville domains.

%%%%%%%%%%%%%%%%%%%%%%%%%%%%%%%%%
% MVZ cotangent bundles
%%%%%%%%%%%%%%%%%%%%%%%%%%%%%%%%%
\subsubsection{Cotangent bundles}
In \cite{MoViNi12}, Monzner, Vichery and Zapolsky show the following.
\begin{thm}
\label{thm:Embedding_R}Let $N$ be a closed manifold. There exists an isometric group embedding
of $(\mathbb{R},d_{\text{st}})$ in $(\Ham_{c}(T^{\ast}N),d_{\gamma})$.
\end{thm}

For the reader's convenience, we give a detailed proof of Theorem
\ref{thm:Embedding_R} along the lines of \cite{MoViNi12}. We refer
the reader to Section \ref{sec:Spectral-invariants} for details on
the standard notation in use here.

Fix $H\in C_{c}^{\infty}(T^{\ast}N)$ such that $H|_{N}=1$ and $0\leq H\leq1$.
Let $\iota:\RR\to\Ham_{c}(T^{\ast}N)$ be the map defined by 
\begin{align*}
\iota(s)=\varphi_{sH}
\end{align*}
where $\varphi_{sH}\in\Ham_{c}(T^{\ast}N)$ is the time-one map associated
to $sH\in C_{c}^{\infty}(T^{\ast}N)$. We claim that $\iota$ is the
desired embedding.

We first bound $d_{\gamma}(\iota(s),\iota(s'))$ from above. As previously
mentioned, if $F\in C_{c}^{\infty}(T^{\ast}N)$, then $\gamma(\varphi_{F})\leq\norm{F}$,
where $\norm{F}=\max\abs{F}$ denotes the $C^{0}$-norm of $F$. Moreover,
since $H$ is autonomous, $sH\#\overline{{s'H}}=(s-s')H$. Therefore,
\begin{align*}
d_{\gamma}(\iota(s),\iota(s')) & =\gamma(\iota(s)\iota(s')^{-1})\leq\norm{(s-s')H}=\abs{s-s'}.
\end{align*}

Now, we bound $d_{\gamma}(\iota(s),\iota(s'))$ from below. In \cite[Section 2]{MoViNi12}
a Lagrangian spectral invariant $\ell(1,\bullet)$ with respect to
the zero section $N$ of $T^{\ast}N$ is constructed. It is shown
that $\ell(1,\bullet)\leq c(1,\bullet)$. One key property of $\ell(1,\bullet)$
is that if $F|_{N}=r$ for a constant $r\in\RR$, then $\ell(1,F)=r$.
Thus, by definition of $H$ and Lemma \ref{lem:c geq 0}, we have
\begin{align*}
d_{\gamma}(\iota(s),\iota(s'))=c(1,(s-s')H)+c(1,(s'-s)H)\geq c(1,(s-s')H).
\end{align*}
and by the properties of $\ell(1,\bullet$), 
\begin{align*}
d_{\gamma}(\iota(s),\iota(s'))\geq\ell(1,(s-s')H)=s-s'.
\end{align*}
In an analogous fashion, we obtain $d_{\gamma}(\iota(s),\iota(s'))\geq s'-s$.
We can therefore conclude that 
\[
d_{\gamma}(\iota(s),\iota(s'))\geq\abs{s-s'}
\]
 as desired. This completes the proof.

Theorem \ref{thm:Embedding_R} immediately implies
\begin{cor}
\label{cor:cotangent_bundle}Let $N$ be a closed manifold. Then $\diam_{\gamma}(DT^{\ast}N)=+\infty$
. 
\end{cor}

To prove Corollary \ref{cor:cotangent_bundle}, one does not need
the Lagrangian spectral norm. Indeed, since the zero section $N\subset DT^{\ast}N$
is an exact closed Lagrangian submanifold, Corollary \ref{cor:cotangent_bundle}
follows directly from Proposition \ref{prop:Viterbo_Lagrangian} and
Theorem \ref{thm:finite_diameter}.

%%%%%%%%%%%%%%%%%%%%%%%%%%%%%%%%%
%%%%%%%%%%%%%%%%%%%%%%%%%%%%%%%%%
% Other symplectic manifolds
%%%%%%%%%%%%%%%%%%%%%%%%%%%%%%%%%
%%%%%%%%%%%%%%%%%%%%%%%%%%%%%%%%%
\subsection{The spectral diameter of other symplectic manifolds}

%%%%%%%%%%%%%%%%%%%%%%%%%%%%%%%%%
% Finiteness of Spectral diameter in literature
%%%%%%%%%%%%%%%%%%%%%%%%%%%%%%%%%
It has been known for a
long time \cite{EnPo03} that for $(\mathbb{C}P^{n},\omega_{\text{FS}})$,
\begin{align*}
\diam_{\gamma}(\mathbb{C}P^{n})\leq\int_{\mathbb{C}P^{1}}\omega_{\text{FS}} & .
\end{align*}
The above upper bound was latter optimized by Kislev and Shelukhin
in \cite[Theorem G]{KiSh18} to 
\begin{align*}
\diam_{\gamma}(\mathbb{C}P^{n})=\frac{n}{n+1}\int_{\mathbb{C}P^{1}}\omega_{\text{FS}}.
\end{align*}
However, for a surface $\Sigma_{g}$ of genus $g\geq1$, the spectral
diameter is infinite. This case is covered by the following theorem
of Kislev and Shelukhin \cite[Theorem D]{KiSh18} which is a sharpening
of a result of Usher \cite[Theorem 1.1]{Us13}. 
\begin{thm}
\label{thm:KisShe}Let $(M,\omega)$ be a closed symplectic manifold
that admits an autonomous Hamiltonian $H\in C^{\infty}(M,\mathbb{{R}})$
such that 
\begin{lyxlist}{00.00.0000}
\item [{{\normalfont \textbf{U1}}}] all the contractible periodic orbits
of $X_{H}$ are constant.
\end{lyxlist}
Then $\diam_{\gamma}(M)=+\infty$. 
\end{thm}

Theorem \ref{thm:KisShe} allows one to prove that the spectral diameter
is infinite in many cases. A list of examples in which condition \textbf{U1}
holds can be found in \cite[Section 1]{Us13}. As mentioned above,
surfaces of positive genus satisfy \textbf{U1}. Also, if $(N,\omega_{N})$
satisfies \textbf{U1} then so does $(M\times N,\omega_{M}\oplus\omega_{N})$
for any other closed symplectic manifold $(M,\omega_{M})$. 

In \cite{Ka20}, Kawamoto proves that the spectral diameter of the
quadrics $Q^{2}$ and $Q^{4}$ (of real dimension 4 and 8 respectively)
and certain stabilizations of them is infinite.

%%%%%%%%%%%%%%%%%%%%%%%%%%%%%%%%%
% Product of symplecticaly aspherical manifoilds
%%%%%%%%%%%%%%%%%%%%%%%%%%%%%%%%%
\subsubsection{Symplectically aspherical manifolds\label{subsec:Symplectically-aspherical-manifo}}

Recall that a symplectic manifold $(M,\omega_{M})$ is symplectically
aspherical if both $\omega_{M}$ and the first Chern class $c_{1}(M)$
of $M$ vanish on $\pi_{2}(M)$, namely, for every continuous map
$f:S^{2}\to M$,
\[
\langle[\omega_{M}],f_{\ast}[S^{2}]\rangle=0=\langle c_{1}(M),f_{\ast}[S^{2}]\rangle.
\]
An open subset $U\subset M$ is said to be incompressible if the map
$\pi_{1}(U)\to\pi_{1}(M)$ induced by the inclusion is injective. 

As pointed out in \cite{BuHuSe21}, it has been conjectured that $\diam_{\gamma}(M)=+\infty$
on all closed symplectically aspherical manifolds. Here, we prove
that conjecture in the case of the twisted product $(M\times M,\omega\oplus-\omega)$
of a closed symplectically aspherical manifold $(M,\omega)$ with
itself. But first, a more general result. 

\begin{customprop}{D}\label{prop:aspherical_inf}
	Let $(M,\omega)$ be a closed symplectically aspherical manifold of dimension $2n$. Suppose there exists an incompressible Liouville domain $D$ of codimension $0$ embedded inside $M$ with $SH^\ast(D)\neq 0$. Then, $\diam_\gamma(M)=+\infty$.
\end{customprop}
\begin{proof}
Let $H$ be a compactly supported Hamiltonian in $D$ and denote by
$\iota:D\to M$ the embedding. By a cohomological analogue of \cite[Claim 5.2]{GaTa20},
we have that

\begin{align*}
c_{D}(\beta,H)=\max_{\substack{\alpha\in\HH^{\ast}(M)\\
\iota^{\ast}(\alpha)=\beta
}
}c_{M}(\alpha,H)
\end{align*}
for all $\beta\in\HH^{\ast}(D)$ where $c_{D}$ and $c_{M}$ are the
spectral invariants on $D$ and $M$ respectively. In particular,
we know that the unit $1_{M}\in\HH^{\ast}(M)$ is sent to the unit
$1_{D}\in\HH^{\ast}(D)$ under the map $\iota^{\ast}:\HH^{\ast}(M)\to\HH^{\ast}(D)$.
Moreover, it is well known that the spectral invariant with respect
to the unit can be implicitly written as
\[
c_{M}(1_{M},H)=\max_{\alpha\in\HH^{\ast}(M)}c_{M}(\alpha,H)
\]
(see Lemma \ref{lem:Def_implicite_c(1,H)}). Therefore, fixing $\beta=1_{D}$,
we have
\begin{align*}
c_{D}(1_{D},H)=c_{M}(1_{M},H).
\end{align*}
Using Theorem \ref{thm:finite_diameter}, the above equation thus
yields the desired result. 
\end{proof}
\begin{customcor}{E}\label{prop:Closed_unbounded}
	Let $(M,\omega)$ be a closed symplectically aspherical manifold. Then, $\diam_\gamma(\Ham(M\times M,\omega\oplus -\omega))=+\infty$.
\end{customcor}
\begin{proof}
Consider the closed Lagrangian given by the diagonal $L=\Delta$ inside
$(M\times M,\omega_{M}\oplus-\omega_{M})$. In virtue of the Weinstein
neighborhood theorem, there exists an open neighborhood $U$ of $L$
and a symplectomorphism $\psi:U\to D_{\varepsilon}T^{\ast}L$ such
that $\varphi(L)$ coincides with the zero section of an $\epsilon$-radius
codisk bundle $D_{\varepsilon}T^{\ast}L$ over $L$. The Liouville
structure on $D_{\varepsilon}T^{\ast}L$ pulls back to a Liouville
structure on $U$. Note that, inside $M\times M$, $L$ is incompressible,
i.e. the map $\pi_{1}(L)\to\pi_{1}(M\times M)$ of first homotopy
groups induced by the inclusion $L\to M\times M$ is injective. Therefore,
by homotopy equivalence, $U$ and $D_{\varepsilon}T^{\ast}L$ are
also incompressible. The desired result follows directly from Proposition
\ref{prop:aspherical_inf}.
\end{proof}

%%%%%%%%%%%%%%%%%%%%%%%%%%%%%%%%%
%%%%%%%%%%%%%%%%%%%%%%%%%%%%%%%%%
% Hofer geometry
%%%%%%%%%%%%%%%%%%%%%%%%%%%%%%%%%
%%%%%%%%%%%%%%%%%%%%%%%%%%%%%%%%%
\subsection{Hofer geometry}
\label{sec:Hofer geometry}

%%%%%%%%%%%%%%%%%%%%%%%%%%%%%%%%%
% A question of Le Roux
%%%%%%%%%%%%%%%%%%%%%%%%%%%%%%%%%
As hinted at above, the finiteness of the spectral diameter plays
a role in Hofer geometry. In particular, it can be used to study the following question posed by Le Roux in \cite{Le10}:

\vspace{10pt}\noindent \textbf{Question 1.} \label{que:Question 1}For
any $A>0$, let 
\[
E_{A}(M,\omega):=\ens{\varphi\in\Ham(M,\omega)\mid d_{H}(\text{Id},\varphi)>A}
\]
be the complement of the closed ball of radius $A$ in Hofer's metric.
For all $A>0$, does $E_{A}(M,\omega)$ have non-empty $C^{0}$-interior?

\vspace{10pt}

\noindent Indeed, in the case of closed symplectically aspherical
manifolds with infinite spectral diameter, a positive answer to Question
1 was given by Buhovsky, Humili\`ere and Seyfaddini (see also \cite{Ka19,Ka20}
for the positive and negative monotone cases).
\begin{thm}[\cite{BuHuSe21}]
\label{thm:BuHuSe21} Let $(M,\omega)$ be a closed, connected and
symplectically aspherical manifold. If $\diam_{\gamma}(M)=+\infty$,
then $E_{A}(M,\omega)$ has non-empty $C^{0}$-interior for all $A>0$.
\end{thm} 

Using Theorem \ref{thm:BuHuSe21} in conjunction with Corollary \ref{prop:Closed_unbounded},
we directly obtain the following answer to Question \ref{que:Question 1}
in the specific setting of Corollary \ref{prop:Closed_unbounded}.

\begin{customcor}{F}\label{prop:EA_product}
	Let $(M,\omega)$ be a closed symplectically aspherical manifold. Then, $E_A(M\times M,\omega\oplus-\omega)$ has a non-empty $C^0$-interior for all $A>0$.
\end{customcor}

\subsubsection*{Acknowledgements}

This research is a part of my PhD thesis at the Université de Montréal
under the supervision of Egor Shelukhin. I thank him for proposing
this project and outlining the approach used to carry it out in this
paper. I am deeply indebted to him for the countless valuable discussions
we had regarding spectral invariants and symplectic cohomology. I
would also like to thank Octav Cornea and François Lalonde for their
comments on an early draft of this project. I thank Leonid Polterovich, Felix Schlenk and Shira Tanny for their comments which helped improve the exposition. Finally, I am grateful to Marcelo Atallah, Filip Brocic, François Charette, Jean-Philippe Chassé, Dustin Connery-Grigg, Jonathan Godin, Jordan Payette and Dominique Rathel-Fournier for fruitful conversations. This research was partially
supported by Fondation Courtois.

\section{Liouville domains and admissible Hamiltonians}

In this subsection we recall the definition of Liouville domains,
specify the class of Hamiltonians we will restrict our attention to
and describe how their Floer trajectories behave at infinity.

\subsection{Completion of Liouville domains\label{subsec:Completion-of-Liouville}}

A Liouville domain $(D,\dd\lambda,Y)$ is an exact symplectic manifold
with boundary on which the vector field $Y$, defined by $Y\iprod\dd\lambda=\lambda$
and called the Liouville vector field, points outwards along $\partial D$.
Denote by $\hat{D}=D\cup[1,\infty)\times\partial D$ the completion
of $D$ and $(r,x)$ the coordinates on $[1,\infty)\times\partial D$.
Here, we glue $\partial D$ and $\ens{1}\times\partial D$ with respect
to the reparametrization $\psi_{Y}^{\ln r}$ of the Liouville flow
generated by $Y$. Given $\delta>0$, let
\begin{align*}
D^{\delta} & =\psi_{Y}^{\ln\delta}(D)=D\setminus(\delta,\infty)\times\partial D.
\end{align*}
We extend the Liouville form $\lambda$ to $\hat{D}$ by defining
$\hat{\lambda}:T\hat{D}\to\RR$ as 
\begin{align*}
\hat{\lambda}\mid_{D}\,=\lambda & \quad\text{and}\quad\hat{\lambda}\mid_{\hat{D}\setminus D}\,=r\alpha
\end{align*}
where $\alpha=\lambda|_{\partial D}$. The cylindrical portion $[1,\infty)\times\partial D$
of $\hat{D}$ is thus equipped with the symplectic form $\omega=\dd(r\alpha)$. 

\begin{figure}[h]
\begin{centering}
\includegraphics{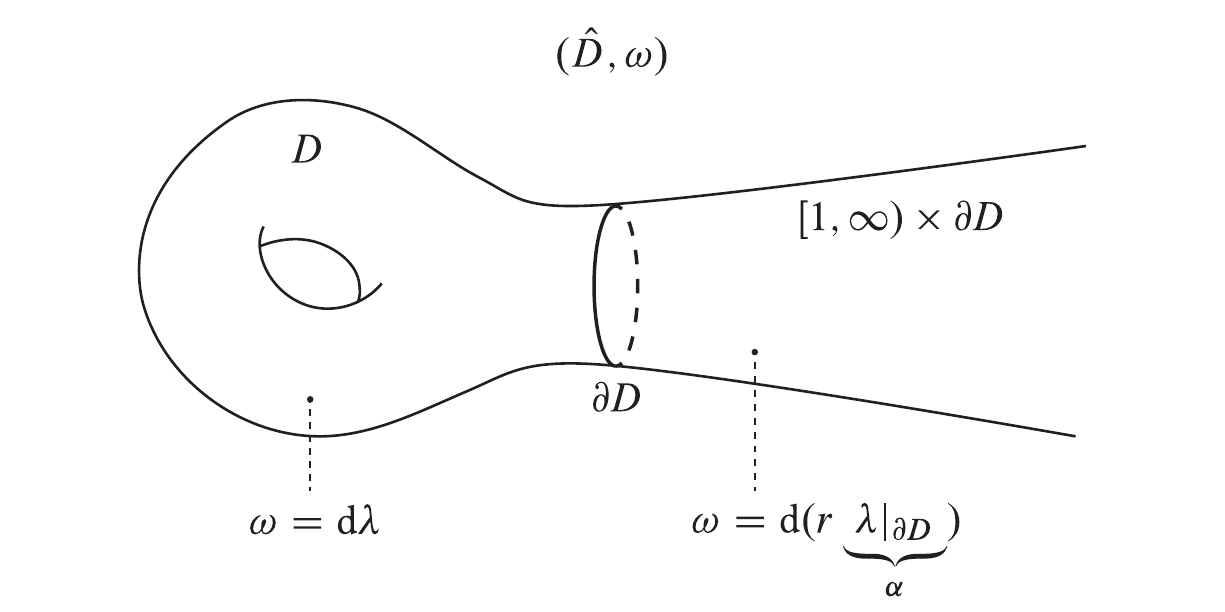}
\par\end{centering}
\caption{A Liouville domain.}
\end{figure}
\noindent The skeleton $\Sk(D)$ of $(D,\dd\lambda,Y)$ is defined
by 
\begin{align*}
\Sk(D)=\bigcap_{0<r<1}\psi_{Y}^{\ln r}(D).
\end{align*}
Denote by $R_{\alpha}$ the Reeb vector field on $\partial D$ associated
to $\alpha$, meaning
\begin{align*}
R_{\alpha}\iprod\dd\alpha=0,\quad & \alpha(R)=1.
\end{align*}
 We define $\Spec(\partial D,\alpha)$ to be the set of periods of
closed characteristics, the periodic orbits generated by $R_{\alpha}$,
on $\partial D$ and put 
\[
T_{0}=\min\Spec(\partial D,\lambda).
\]
As a subset of $\RR$, $\Spec(\partial D,\alpha)$ is known to be
closed and nowhere dense. For any $A\in\RR$, let $\eta_{A}$ denote
the distance between $A$ and $\Spec(\partial D,\lambda)$. 

\subsection{Admissible Hamiltonians and almost complex structures\label{subsec:Admissible-Hamiltonians-and}}

\subsubsection{Periodic orbits and action functional}

Given a Hamiltonian $H:S^{1}\times\hat{D}\to\RR$, one defines its
time-dependent Hamiltonian vector field $X_{H}^{t}:\hat{D}\to T\hat{D}$
by 
\begin{align*}
X_{H}^{t} & \iprod\omega=-dH_{t}
\end{align*}
where $H_{t}(p)=H(t,p)$. We denote by $\varphi_{H}^{t}:\hat{D}\to\hat{D}$
the flow generated by $X_{H}^{t}$. The set of all contractible 1-periodic
orbits of $\varphi_{H}^{t}$ is denoted by $\mathcal{P}(H)$. An orbit
$x\in\mathcal{P}(H)$ is said to be \emph{non-degenerate} if 
\begin{align*}
\det(\id-d_{x(0)}\varphi_{H}^{1})\neq0
\end{align*}
and \emph{transversally non-degenerate }if the eigenspace associated
to the eigenvalue 1 of the map $d_{x(0)}\varphi^{1}$ is of dimension
1. If all elements of $\mathcal{P}(H)$ are non-degenerate or transversally
non-degenerate, we say that $H$ is regular. 

Let $\mathcal{L}\hat{D}$ be the space of contractible loops in $\hat{D}.$
For a Hamiltonian $H:S^{1}\times\hat{D}\to\RR$ , the \emph{Hamiltonian
action functional} $\mathcal{A}_{H}:\mathcal{L}\hat{D}\to\RR$ associated
to $H$ is defined as 
\begin{align*}
\mathcal{A}_{H}(x)= & \int_{0}^{1}x^{\ast}\hat{\lambda}-\int_{0}^{1}H_{t}(x(t))\:\dd t.
\end{align*}
It is well known that the elements elements of $\mathcal{P}(H)$ correspond
to the critical points of $\mathcal{A}_{H}$, see \cite[section 6]{AuDa14}.
The image of $\mathcal{P}(H)$ under the Hamiltonian action functional
is called the \emph{action spectrum of $H$} and is denoted by $\Spec(H)$. 

\subsubsection{Admissible Hamiltonians}

The completion of a Liouville domain is obviously non-compact. We
thus need to control the behavior at infinity of Hamiltonians we use
in order for them to have finitely many 1-periodic contractible orbits. 
\begin{defn}
\label{def:Affine_hamiltonian}Let $r_{0}>1$. A Hamiltonian $H$
is $r_{0}$-admissible if 

\begin{itemize}
\item $H(t,x,r)=h(r)$ on $\hat{D}\setminus D$,
\item $h(r)$ is $C^2$-small on $(1,r_0)$,
\item $h(r) = \tau_H r-\tau_Hr_0$ on $(r_0,+\infty)$ for $\tau_H\in(0,\infty)\setminus\Spec(\del D, \alpha)$,
	\item $H$ is regular.
\end{itemize}We denote the set of such Hamiltonians $\mathcal{H}_{r_{0}}$. 

\begin{figure}[H]
\begin{centering}
\includegraphics{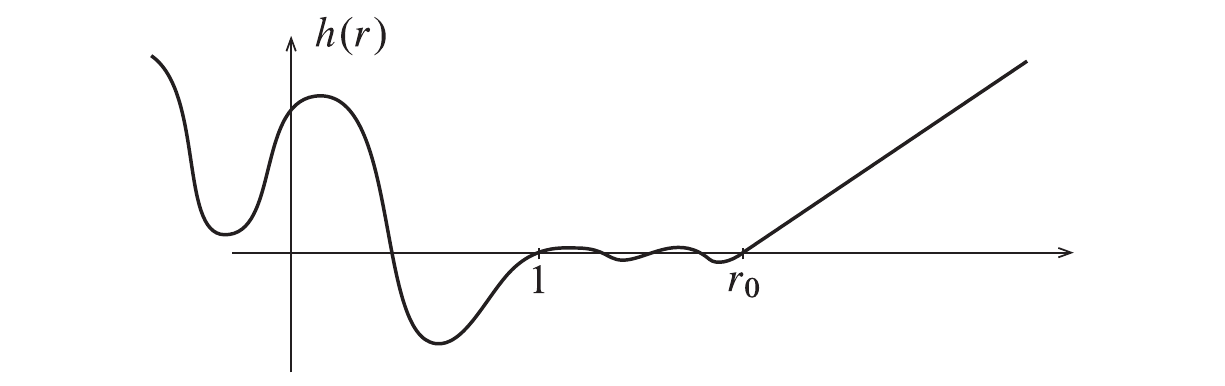}
\par\end{centering}
\caption{An $r_{0}$-admissible Hamiltonian.}
\end{figure}

\noindent We will also consider the set $\mathcal{H}_{r_{0}}^{0}\subset\mathcal{H}_{r_{0}}$
of $r_{0}$-admissibe Hamiltonians which are negative on $D$. In
some cases, it is not necessary to specify $r_{0}$ as long as it
is greater than $1$. For that purpose, we define 
\begin{align*}
\mathcal{H}=\bigcup_{r_{0}>1}\mathcal{H}_{r_{0}},\quad\mathcal{H}^{0}=\bigcup_{r_{0}>1}\mathcal{H}_{r_{0}}^{0}.
\end{align*}
\end{defn}

\begin{rem}
\label{rem:1-per_orb_S1_fam}Suppose $H\in\mathcal{H}$. If $x\in\mathcal{P}(H)\cap\hat{D}\setminus D$
is non constant, then it is necessarily transversally non-degenerate.
Indeed, since $H$ is time-independent there by definition, for any
$c\in\RR$, $x(t-c)$ is also a 1-periodic orbit of $H$. 
\end{rem}

\begin{lem}
\label{lem:Finite_periodic_orbits}If $H\in\mathcal{H}$, then $\abs{\mathcal{P}(H)}$
consists of a finite number of periodic orbits and $S^{1}$ families
of periodic orbits.
\end{lem}

\begin{proof}
Since $D$ is compact, there is a finite number of 1-periodic orbits
of $H$ inside it. 

Next, we look at the elements of $\mathcal{P}(H)$ which sit inside
$\hat{D}\setminus D$. On this subset of $\hat{D}$, we know that
$H=h(r)$ and $\omega=\dd\hat{\lambda}$. Therefore, on $\hat{D}\setminus D$
\begin{align*}
X_{H}\iprod\omega & =X_{H}\iprod(\dd r\wedge\alpha+r\dd\alpha)\\
 & =\dd r(X_{H})\alpha-\alpha(X_{H})\dd r+rX_{H}\iprod\dd\alpha
\end{align*}
and $\dd H=h'(r)\dd r$. Hamilton's equation thus yields 
\begin{align*}
\dd r(X_{H})=0=X_{H}\iprod\dd\alpha,\quad\alpha(X_{H})=h'(r).
\end{align*}
The three equations above imply the following two facts, 
\begin{itemize}
\item on $\hat{D}\setminus D$, $X_{H}=h'(r)R_{\alpha}$;
\item if $x\in\mathcal{P}(H)$ is such that $x\cap\hat{D}\setminus D\neq\varnothing$,
then $x\subset\ens{r}\times\partial D$ for some $r>1$.
\end{itemize}
We conclude that a 1-periodic orbit $x$ of $H$ which lies inside
$\{r\}\times\partial D$ corresponds to a Reeb orbit of period $h'(r)$.
Notice that since $\tau_{H}\notin(0,+\infty)\cap\Spec(\partial D,\alpha)$,
$\mathcal{P}(H)\cap(\hat{D}\setminus D)=\mathcal{P}(H)\cap(D^{r_{0}}\setminus D)$.
Moreover, $\Spec(\partial D,\alpha)$ being nowhere dense and closed
in $\RR$, we have that $\Spec(\partial D,\alpha)\cap(1,r_{0})$ is
a finite set. We can therefore conclude, since every non-constant
1-periodic orbit of $H$ in $\hat{D}\setminus D$ is transversally
non-degenerate, that $\mathcal{P}(H)\cap(\hat{D}\setminus D)$ consists
of a finite number of $S^{1}$ families of periodic orbits. 
\end{proof}
\begin{rem}
\label{rem:Action_formula_cylindrical_part}The fact that admissible
Hamiltonians are radial on the cylindrical part of $\hat{D}$ allows
us to express the action of the 1-periodic orbits inside $\hat{D}\setminus D$
in terms of that radial function. To see this, we fix $H\in\mathcal{H}$
and compute the action of a non constant orbit $x\in\mathcal{P}(H)\cap(\hat{D}\setminus D)$
which we suppose lies inside $\{r\}\times\partial D$ for $r>1$:
\begin{align*}
\mathcal{A}_{H}(x) & =\int_{0}^{1}x^{\ast}\hat{\lambda}-\int_{0}^{1}H\comp x\:\dd t\\
 & =\int_{0}^{1}r\alpha(X_{H})\dd t-\int_{0}^{1}h(r)\dd t=rh'(r)-h(r).
\end{align*}
The function $A_{H}(r)=rh'(r)-h(r)$ on the right hand side of the
above equation has a nice geometric interpretation. On the graph of
$h$, $A_{H}(r')$ corresponds to minus the $y$-coordinate of the
intersection of the tangent at the point $(r',h(r'))$ and the $y$-axis. 
\end{rem}

\begin{figure}[H]

\begin{centering}
\includegraphics{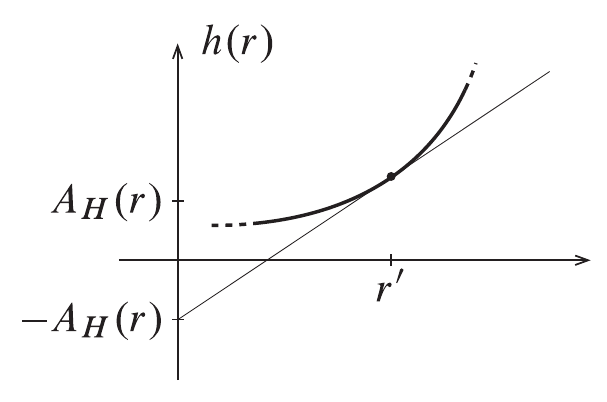}
\par\end{centering}
\caption{Action value of a periodic orbit contained in $\{r'\}\times\partial D$.}
\end{figure}

\subsubsection{Monotone homotopies}

We will need to also restrict the types of Hamiltonian homotopies
we consider to the following class. 
\begin{defn}
Let $H_{s}=\ens{H_{s}}_{s\in\RR}$ be a smooth homotopy from $H_{+}\in\mathcal{H}_{r_{0}}$
to $H_{-}\in\mathcal{H}_{r_{0}'}$ We say that $H_{s}$ is a \emph{monotone
homotopy }if the following conditions hold 

\begin{itemize}
	\item  $\exists S>0$ such that $H_{s'}=H_-$ for $s'<-S$ and $H_{s'}=H_+$ for $s'>S$, 
	\item $H_s = h_s(r)$ on $\hat{D}\setminus D$,
	\item for $R=\max \{r_0,r_0'\}$, $h_s(r) = \tau_sr+\eta_s$ on $(R,+\infty)$ for smooth functions $\tau_s$ and $\eta_s$ of $s$ and $\partial_s \tau_s\leq 0$.
\end{itemize}
\end{defn}

For $H_{+}\in\mathcal{H}_{r_{0}}$ and $H_{-}\in\mathcal{H}_{r_{0}'}$
with $\tau_{H_{+}}=\tau_{+}\leq\tau_{-}=\tau_{H_{-}}$, we can explicitly
construct a monotone homotopy in the following way. Fix a positive
constant $S>0$. Let $\beta:\RR\to[0,1]$ be a smooth function such
that $\beta(s)=0$ for $s\leq-S$, $\beta(s)=1$ for $s\geq S$ and
$\beta'(s)>0$ for all $s\in(-S,S)$. Define 
\begin{align*}
H_{s} & =H_{-}+\beta_{s}(H_{+}-H_{-}).
\end{align*}
For $R=\max\ens{r_{0},r_{0}'}$ we have, on $\hat{D}\setminus D^{r_{0}}$,
\begin{align*}
H_{s}(t,r,p) & =(\beta(s)(\tau_{+}-\tau_{-})+\tau_{-})r+\beta(s)(\eta_{+}-\eta_{-})+\eta_{-}
\end{align*}
and $\partial_{s}\tau_{s}\leq0$ as desired. 

\subsubsection{Admissible almost complex structures}

Let $J$ be an almost complex structure on $\hat{D}$. Recall that
$J$ is $\omega$-compatible if the map $g_{J}:TM\otimes TM\to\RR$
defined by 
\begin{align*}
g_{J}(v,w) & =\omega(v,Jw)
\end{align*}
is a Riemannian metric. To control the behavior of $\omega$-compatible
almost complex structures at infinity, we make the following definition. 
\begin{defn}
Let $J$ be an $\omega$-compatible almost complex structure on $\hat{D}$.
We say that $J$ is \emph{admissible} if $J_{1}=J|_{\hat{D}\setminus D}$
is of \emph{contact type}. More specifically, we ask that 
\begin{align*}
J_{1}^{\ast}\hat{\lambda} & =\dd r.
\end{align*}
We denote the set of such almost complex structures by $\mathcal{J}$.
A pair $(H,J)$ where $H\in\mathcal{H}_{r_{0}}$ and $J\in\mathcal{J}$
is called an \emph{$r_{0}$-admissible pair}. 
\end{defn}

\subsection{Floer trajectories and maximum principle. }

\noindent In this subsection, we recall some analytical aspects of
Floer theory on Liouville domains. Issues regarding transversality
will be dealt with in the next section. 

\subsubsection{Floer trajectories}

Consider an Hamiltonian $H:S^{1}\times\hat{D}\to\RR$ and two 1-periodic
orbits $x_{\pm}\in\mathcal{P}(H)$. Let $J$ be an $\omega$-compatible
almost complex structure on $\hat{D}.$ A \emph{Floer trajectory}
between $x_{-}$ and $x_{+}$ is a solution $u:\RR\times S^{1}\to\hat{D}$
to the \emph{Floer equation 
\begin{align*}
\partial_{s}u+J(\partial_{t}u-X_{H}) & =0
\end{align*}
}that converges uniformly in $t$ to $x_{-}$ and $x_{+}$ as $s\to\pm\infty$:
\begin{align*}
\lim_{s\to\pm\infty}u(s,t)=x_{\pm}.
\end{align*}
We denote the moduli space of such trajectories $\mathcal{M}'(x_{-},x_{+};H).$
We may reparametrize a solution $u\in\mathcal{M}'(x_{-},x_{+};H)$
in the $\RR$-coordinate by adding a constant. Thus, Floer trajectories
occur in $\RR$-families. The space of unparametrized solutions is
denoted by $\mathcal{M}(x_{-},x_{+};H)=\mathcal{M}'(x_{-},x_{+};H)/\RR$.
When the context is clear, we will drop $H$ from the notation and
simply write $\mathcal{M}(x_{-},x_{+})$. 

If we replace $H$ with a monotone homotopy $H_{\bullet}=\ens{H_{s}}_{s\in\mathbb{R}}$,
we can instead consider solutions $u:\RR\times S^{1}\to\hat{D}$ to
the $s$-\emph{dependent Floer equation 
\begin{align*}
\mathcal{\partial}_{s}u+J(\partial_{t}u-X_{H_{s}})=0
\end{align*}
}that converge uniformly in $t$ to $x_{\pm}\in\mathcal{P}(H_{\pm})$
as $s\to\pm\infty$. The moduli space of such trajectories is denoted
by $\mathcal{\mathcal{M}}(x_{-},x_{+};H_{\bullet})$. Unlike the $s$-independent
case, $\mathcal{M}(x_{-},x_{+};H_{\bullet})$ does not admit a free
$\RR$-action by which we can quotient. 

\subsubsection{Maximum principle}

To define Floer cohomology of $\hat{D}$, we need to control the behavior
of the Floer trajectories. In particular, we have to make sure they
do not escape to infinity. Admissible Hamiltonians and admissible
complex structures allow us to achieve that requirement. The first
result in that direction is the maximum principle for Floer trajectories.
In what follows we say that $v$ is a local Floer solution of $(H,J)$
in $\hat{D}\setminus D$ if 
\begin{align*}
v=u\big|_{u^{-1}(\im u\cap\hat{D}\setminus D)}:u^{-1}(\im u\cap\hat{D}\setminus D)\to\hat{D}\setminus D
\end{align*}
for some $u\in\mathcal{M}(x_{-},x_{+};H)$.
\begin{lem}[Generalized maximum principle \cite{Vi99}]
Let $(H,J)$ be an $r_{0}$-admissible pair on $\hat{D}$. Suppose
$v$ is a local Floer solution of $(H,J)$ in $\hat{D}\setminus D$.
Then, the $r$-coordinate $r\comp v$ of $v$ does not admit an interior
maximum unless $r\comp v$ is constant. 
\end{lem}

\begin{rem}
The generalized maximum principle still holds if we replace $H\in\mathcal{H}$
by a monotone homotopy $H_{s}$ between $H_{+}\in\mathcal{H}_{r_{0}}$
and $H_{-}\in\mathcal{H}_{r_{0}'}$ and if $v$ is a local solution
of the $s$-dependent Floer equation 
\begin{align*}
\partial_{s}v+J(\partial_{t}v-X_{H_{s}}) & =0
\end{align*}
inside $\hat{D}\setminus D$. 

From the maximum principle above, we immediately obtain the following
corollary which guarantees that Floer trajectories do not escape to
infinity. 
\end{rem}

\begin{cor}
\label{cor:No_escape}Let $(H,J)$ be an $r_{0}$-admissible pair
on $\hat{D}$ and let $x_{\pm}\in\mathcal{P}(H)$. If $u\in\mathcal{M}(x_{-},x_{+})$,
then
\begin{align*}
\im u\subset D^{R},\quad\text{for }R=\max\{r\comp x_{-},r\comp x_{+},r_{0}\}.
\end{align*}
If $H_{s}$ is a monotone homotopy between $H_{-}\in\mathcal{H}_{r_{0}}$
and $H_{+}\in\mathcal{H}_{r_{0}'}$ and $u$ is a solution to the
$s$-dependent Floer equation between $x_{-}\in\mathcal{P}(H_{-})$
and $x_{+}\in\mathcal{P}(H_{+})$, then 
\begin{align*}
\im u\subset D^{R},\quad\text{for }R=\max\{r\comp x_{-},r\comp x_{+},r_{0},r_{0}'\}.
\end{align*}
\end{cor}

\subsubsection{Energy}

An important quantity which is associated to a Floer trajectory is
its \emph{energy. }It is defined as
\begin{align*}
E(u) & =\frac{1}{2}\int_{\RR\times S^{1}}\left(\abs{\partial_{s}u}_{J}^{2}+\abs{\partial_{t}u-X_{H}}_{J}^{2}\right)\dd s\wedge\dd t
\end{align*}
where $\abs{\,\cdot\,}_{J}$ is the norm corresponding to $g_{J}$.
Using the Floer equation, we can write 
\begin{align*}
\abs{\partial_{t}u-X_{H}}_{J}^{2} & =\omega(J\partial_{s}u,-\partial_{s}u)=\omega(\partial_{s}u,J\partial_{s}u)=\abs{\partial_{s}u}_{J}^{2}.
\end{align*}
Thus, the energy can be written more compactly as 
\begin{align*}
E(u) & =\int_{\RR\times S^{1}}\abs{\partial_{s}u}_{J}^{2}\,\dd s\wedge\dd t.
\end{align*}

It is often useful to estimate the difference in Hamiltonian action
of the ends of a Floer trajectory in terms of the energy of that trajectory.
This can be achieved using the maximum principle and Stokes Theorem. 
\begin{lem}
\label{lem:energy_act_estimates}Let $(H,J)$ be an $r_{0}$-admissible
pair and let $u\in\mathcal{M}'(x_{-},x_{+};H)$ for $x_{\pm}\in\mathcal{P}(H)$.
Then, 
\begin{align*}
0\leq E(u)=\mathcal{A}_{H}(x_{+})-\mathcal{A}_{H}(x_{-}).
\end{align*}
If $H_{s}$ is a monotone homotopy between $H_{+}\in\mathcal{H}_{r_{0}}$
and $H_{-}\in\mathcal{H}_{r_{0}'}$ that is constant in the $s$-coordinate
for $s>\abs{S}$ then 
\begin{align*}
0\leq E(u)\leq\mathcal{A}_{H_{+}}(x_{+})-\mathcal{A}_{H_{-}}(x_{-})+\sup_{\substack{s\in[-S,S],\\
t\in S^{1},p\in D^{R}
}
}\partial_{s}H_{s}(t,p)
\end{align*}
where $R=\max\{r\comp x_{-},r\comp x_{+},r_{0},r_{0}'\}$.
\end{lem}

\section{Filtered Floer and symplectic cohomology}

We present in this subsection a brief overview of Floer cohomology
for completions of Liouville domains and their symplectic cohomology.
For more details we refer the reader to \cite{CiFlHo95}, \cite{CiFlHoWy96},
\cite{Vi99}, \cite{We06}, \cite{CiFrOa10} and \cite{Ri13}.

\subsection{Filtered Floer Cohomology}

\subsubsection{The Floer cochain complex}

Let $(H,J)$ be an admissible pair. As mentioned in Remark \ref{rem:1-per_orb_S1_fam},
the 1-periodic orbits of $H$ on $\hat{D}\setminus D$ come in a finite
number of $S^{1}$-families which we denote by $\hat{x}_{i}$. To
break each $\hat{x}_{i}$ in a finite number of isolated periodic
orbits, we first choose an open neighborhoods $U_{i}$ of each $\hat{x}_{i}$
such that $U_{i}\cap U_{j}=\varnothing$ for $i\neq j$. Then, we
define on each $\hat{x}_{i}$ a Morse function $f_{i}$ having exactly
two critical points : one of index $0$ and another of index $1$.
We extend each $f_{i}$ to its corresponding $U_{i}$. When added
to $H$, these perturbations, which can be chosen as small as we want,
break each of the $S^{1}$-families into two critical points. In virtue
of the action formula derived in Remark \ref{rem:Action_formula_cylindrical_part},
the actions of the new critical points are as close as we want to
the action of their original $S^{1}$-family. We denote by $H_{1}$
the Hamiltonian resulting from this procedure. By abuse of notation
we will write $\mathcal{P}(H)$ for the set of 1-periodic orbits of
$H_{1}$. 

We define the \emph{Floer cochain group of $H$} as the $\mathbb{Z}_{2}$-vector
space \footnote{We use $\mathbb{Z}_2$ coefficients here for simplicity but the cohomological construction that follows  can be carried out with any coefficient ring.}
\begin{align*}
\CF^{\ast}(H) & =\bigoplus_{x\in\mathcal{P}(H)}\mathbb{Z}_{2}\expval{x}.
\end{align*}
As the notation above suggests, $\CF^{\ast}(H)$ is in fact a \emph{graded}
$\mathbb{Z}_{2}$-vector space. Assuming that the first Chern class
$c_{1}(\omega)\in\HH^{2}(\hat{D};\mathbb{Z})$ of $(T\hat{D},J)$
vanishes on $\pi_{2}(\hat{D})$, the Conley-Zehnder index $\CZ(x)\in\mathbb{\mathbb{Z}}$
of a 1-periodic orbit $x\in\mathcal{P}(H)$ is well defined \cite{SaZe92}.
We can therefore equip $\CF^{\ast}(H)$ with the degree 
\begin{align*}
\abs{x}=\frac{\dim\hat{D}}{2}-\CZ(x)
\end{align*}
and define 
\begin{align*}
\CF^{k}(H)=\bigoplus_{\substack{x\in\mathcal{P}(H)\\
\abs{x}=k
}
} & \mathbb{Z}_{2}\expval{x}.
\end{align*}
Here, $\CZ$ is normalized such that for a $C^{2}$-small time-independent
admissible Hamiltonian $F$, 
\begin{align*}
\CZ(x) & =\frac{\dim\hat{D}}{2}-\ind(x)
\end{align*}
where $\ind(x)$ corresponds to the Morse index of $x\in\Crit(F)=\mathcal{P}(F)$.
In particular, if $x$ is a local minimum of $F$, then $\abs{x}=0$.
This convention therefore ensures that the cohomological unit has
degree zero.

For a generic perturbation of $J$, the space $\mathcal{M}(x_{-},x_{+};H)$
is a smooth manifold of dimension 
\begin{align*}
\dim\mathcal{M}(x_{-},x_{+};H)=\CZ(x_{+})-\CZ(x_{-})-1.
\end{align*}
In the case where $|x_{-}|=|x_{+}|+1$, Corollary \ref{cor:No_escape}
and Lemma \ref{lem:energy_act_estimates} allow us to use the standard
compactness arguments, as in \cite[ Chapter 8]{AuDa14} to show that
$\mathcal{M}(x_{-},x_{+};H)$ is a compact manifold of dimension 0.
Knowing that, we define the co-boundary operator $\partial:\CF^{k}(H)\to\CF^{k+1}(H)$
by 
\begin{align*}
\partial x_{+}=\sum_{\abs{x_{-}}=k+1}\#_{2}\mathcal{M}(x_{-},x_{+};H)x_{-}
\end{align*}
where $\#_{2}\mathcal{M}(x_{-},x_{+};H)$ is the count modulo 2 of
components in $\mathcal{M}(x_{-},x_{+},H)$. 

\begin{figure}[H]

\begin{centering}
\includegraphics{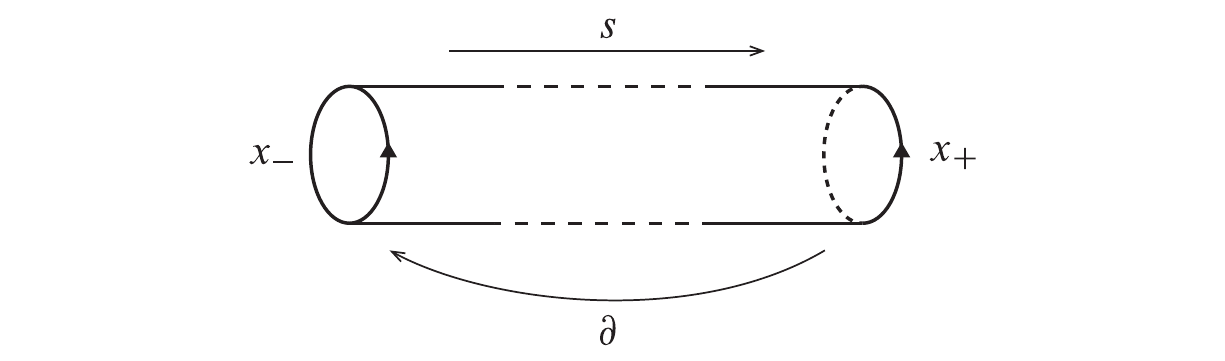}
\par\end{centering}
\caption{The differential in Floer cohomology goes from right to left.}
\end{figure}

\noindent Using once again Corollary \ref{cor:No_escape}, $\partial\comp\partial=0$
holds by standard arguments which appear in \cite[ Chapter 9]{AuDa14}.
The pair $(\CF^{\ast}(H),\partial)$ is thus a graded cochain complex
that we call the \emph{Floer cochain complex of $H$. }

\subsubsection{Filtered Floer cochain complex}

The Hamiltonian action functional induces a filtration on the Floer
cochain complex. For $a\in(\RR\cup\ens{\pm\infty})\setminus\Spec(H)$,
we define 
\begin{align*}
\CF_{<a}^{k}(H) & =\bigoplus_{\substack{x\in\mathcal{P}(H)\\
|x|=k,\ \mathcal{A}_{H}(x)<a
}
}\mathbb{Z}_{2}\expval{x}.
\end{align*}
By definition, we have $\CF^{\ast}(H)=\CF_{<+\infty}^{\ast}(H)$.
Lemma \ref{lem:energy_act_estimates} assures that $\partial$ decreases
the action. Thus, the restriction $\partial_{<a}:\CF_{<a}^{k}(H)\to\CF_{<a}^{k+1}(H)$
of the co-boundary operator is well defined and $(\CF_{<a}^{\ast}(H),\partial_{<a})$
is a sub-complex of $(\CF^{\ast}(H),\partial)$. Now, for $a,b\in(\RR\cup\ens{\pm\infty})\setminus\Spec(H)$
such that $a<b$, we can define the Floer cochain complex in the action
window $(a,b)$ as the quotient 
\begin{align*}
\CF_{(a,b)}^{\ast}(H) & =\frac{\CF_{<b}^{\ast}(H)}{\CF_{<a}^{\ast}(H)},
\end{align*}
on which we denote the projection of the co-boundary operator by 
\begin{align*}
\partial_{(a,b)} & :\CF_{(a,b)}^{k}(H)\to\CF_{(a,b)}^{k+1}(H).
\end{align*}
 Therefore, for $a,b,c\in(\RR\cup\ens{\pm\infty})\setminus\Spec(H)$
such that $a<b<c$, we have an inclusion and a projection 
\begin{align*}
\iota_{a,a}^{b,c} & :\CF_{(a,b)}^{\ast}(H)\to\CF_{(a,c)}^{\ast}(H),\quad\pi_{a,b}^{c,c}:\CF_{(a,c)}^{\ast}(H)\to\CF_{(b,c)}^{\ast}(H)
\end{align*}
 that produce the short exact sequence 

\[
\begin{tikzcd}[column sep = 30pt, row sep = 30pt]
	0 \arrow{r} &
	\CF^\ast_{(a,b)}(H) \arrow{r}{\iota_{a,a}^{b,c}} &
	\CF^\ast_{(a,c)}(H) \arrow{r}{\pi_{a,b}^{c,c}} &
	\CF^\ast_{(b, c)}(H) \arrow{r} &
	0.
\end{tikzcd}
\]For simplicity, we define $\iota^{<c}=\iota_{-\infty,-\infty}^{+\infty,c}$
and $\pi_{>b}=\pi_{-\infty,b}^{+\infty,+\infty}$.

\subsubsection{Filtered Floer cohomology\label{subsec:Filtered-Floer-cohomology}}

Let $a,b\in(\RR\cup\ens{\pm\infty})\setminus\Spec(H)$ such that $a<b$.\emph{
}The above filtered cochain complexes allow us to define the \emph{Floer
cohomology group of $H$ }in the action window\emph{ $(a,b)$ as 
\begin{align*}
\Hf_{(a,b)}^{\ast}(H) & =\frac{\ker\partial_{(a,b)}}{\im\partial_{(a,b)}}.
\end{align*}
}The full Floer cohomology group of $H$ is defined as $\Hf^{\ast}(H)=\Hf_{(-\infty,+\infty)}^{\ast}(H)$.
For $a,b,c\in(\RR\cup\ens{\pm\infty})\setminus\Spec(H)$ such that
$a<b<c$, the short exact sequence on the cochain level induces a
long exact sequence in cohomology: \[
\begin{tikzcd}[column sep = 25pt, row sep = 30pt]
	\Hf^\ast_{(a,b)}(H) \arrow{dr}{[\iota_{a,a}^{b,c}]} &	 
\\
&\Hf^\ast_{(a,c)}(H) \arrow{dl}{[\pi_{a,b}^{c,c}]}
\\
\Hf^\ast_{(b, c)}(H) \arrow{uu}{[+1]}
	&
\end{tikzcd}
\]

For $C^{2}$-small admissible Hamiltonians with small slope at infinity,
the Floer cohomology recovers the standard cohomology of $D$. 
\begin{lem}[{\cite[Section 15.2]{Ri13}}]
\label{lem:Standard_cohomology}Let $H\in\mathcal{H}$ be a $C^{2}$-small
Hamiltonian with $\tau_{H}<T_{0}$ for $T_{0}=\min\Spec(\partial D,\lambda)$.
Then, we have an isomorphism 
\begin{align*}
\Phi_{H}:\HH^{\ast}(D)\to\Hf^{\ast}(H).
\end{align*}
\end{lem}

\begin{rem}
\label{rem:Unit}We can endow $\Hf^{\ast}(H)$ with a ring structure
\cite{Ri13} where the product is given by the pair of pants product.
The unit in $\Hf^{\ast}(H)$, which we denote $1_{H}$, coincides
with $\Phi_{H}(e_{D})$ where $e_{D}$ is the unit in $\HH^{\ast}(D)$
. 
\end{rem}

\subsubsection{Compactly supported Hamiltonians}

We can define the Floer cohomology of compactly supported Hamiltonians
on Liouville domains by first extending to affine functions on the
cylindrical portion of $\hat{D}$. 
\begin{defn}
Denote by $\mathcal{C}(D)$ the set of Hamiltonians with support in
$S^{1}\times(D\setminus\partial D)$. Let $H\in\mathcal{C}(D)$ .
For $\tau\in(0,\infty)\setminus\Spec(\partial D,\lambda)$, we define
the $\tau$-extension $H^{\tau}\in\mathcal{H}_{1}$ of $H$ as follows.
Fix $0<\varepsilon<1$, 

\begin{itemize}
	\item  $H^\tau = H$ on $D$,
	\item $H^\tau =  h_\varepsilon(r)$ on $\hat{D}\setminus D$,
	\item $h_\varepsilon(r)$ is convex for $r\in[1,1+\varepsilon]$ with $h_\varepsilon^{(k)}(1)=0$ for all $k\geq 0$, $h_\varepsilon'(1+\varepsilon) = \tau$ and $h_\varepsilon^{(\ell)}(1+\varepsilon) = 0$ for all $\ell >0$,
	\item $h_\varepsilon(r) = \tau(r-(1+\varepsilon/2))$ for $r\in[1+\varepsilon,+\infty)$.
\end{itemize}

\noindent The Floer cohomology of $H$ is defined as 
\begin{align*}
\Hf_{(a,b)}^{\ast}(H)=\Hf_{(a,b)}^{\ast}(H^{\tau})
\end{align*}
where $0<\tau<T_{0}$. 

\begin{figure}[H]
\begin{centering}
\includegraphics{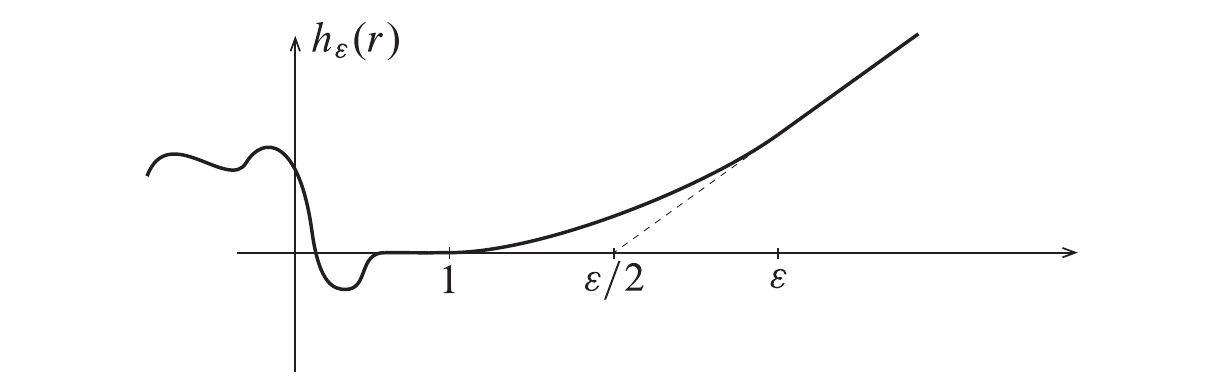}
\par\end{centering}
\caption{The $\tau$-extension of a compactly supported Hamiltonian.}
\end{figure}
\end{defn}

\noindent Since we take a slope $\tau$ smaller than the minimum
Reeb period to define $\Hf_{(a,b)}^{\ast}(H)$, the above definition
doesn't depend on the choice of $\tau$ as we will see in Lemma \ref{lem:Same_Slope}
below.

\subsubsection{Continuation maps\label{subsec:Continuation-maps}}

Let $K\in\mathcal{H}_{r_{0}}$ and $F\in\mathcal{H}_{r_{0}'}$ such
that $F\leq K$. Consider a monotone homotopy $H_{\bullet}$ from
$F$ to $K$. Then from Corollary \ref{cor:No_escape} and Lemma \ref{lem:energy_act_estimates}
in the case of homotopies, we can apply the techniques shown in \cite[ Chapter 11]{AuDa14}
to show that, for $x_{-}\in\mathcal{P}(K)$ and $x_{+}\in\mathcal{P}(F)$
with $|x_{-}|=|x_{+}|$, $\mathcal{M}(x_{-},x_{+};H_{\bullet})$ is
a smooth compact manifold of dimension zero. The \emph{continuation
map $\Phi^{H_{\bullet}}:\CF^{k}(F)\to\CF^{k}(K)$ induced by $H_{s}$
}on the cochain level is defined as 
\begin{align*}
\Phi^{H_{\bullet}}(x_{+})=\sum_{|x_{-}|=k}\#_{2}\mathcal{M}(x_{-},x_{+};H_{\bullet})x_{-}
\end{align*}
where $\#_{2}\mathcal{M}(x_{-},x_{+};H_{\bullet})$ is the count modulo
2 of components in $\mathcal{M}(x_{-},x_{+};H_{\bullet})$. The map
\begin{align*}
[\Phi^{H_{\bullet}}]:\Hf^{\ast}(F)\to\Hf^{\ast}(K)
\end{align*}
is independent of the chosen monotone homotopy and we can denote it
by $[\Phi^{K,F}]$. Consider the monotone homotopy 
\begin{align*}
H_{s}=K+\beta(s)(F-K)
\end{align*}
described in Section \ref{subsec:Admissible-Hamiltonians-and}. We
note that $\partial_{s}H_{s}\leq0$ since $F\leq K$ and $\beta'\geq0$.
Thus the action estimate given by Lemma \ref{lem:energy_act_estimates}
for homotopies yields 
\begin{align*}
\mathcal{A}_{K}(x_{-})\leq\mathcal{A}_{H}(x_{+})+\sup_{\substack{s\in[-S,S],\\
t\in S^{1},p\in D^{R}
}
}\partial_{s}H_{s}(t,p)\leq\mathcal{A}_{H}(x_{+})
\end{align*}
for $x_{-}\in\mathcal{P}(K)$ and $x_{+}\in\mathcal{P}(F)$. Therefore,
the continuation map decreases the action and hence induces maps 
\begin{align*}
[\Phi_{(a,b)}^{K,F}]:\Hf_{(a,b)}^{\ast}(F)\to\Hf_{(a,b)}^{\ast}(K)
\end{align*}
 that commute with the inclusion and restriction maps as follows \cite[Section 8]{Ri13}:

\begin{equation}
\label{cont_proj_res_comm}
\begin{tikzcd}[column sep = 25pt, row sep = 30pt]
	\cdots \arrow{r} &
	\Hf^\ast_{(a,b)}(F) \arrow{r}{[\iota_{a,a}^{b,c}]}\arrow{d}{[\Phi_{(a,b)}^{K,F}]} &
	\Hf^\ast_{(a,c)}(F) \arrow{r}{[\pi_{a,b}^{c,c}]}\arrow{d}{[\Phi_{(a,c)}^{K,F}]} &
	\Hf^\ast_{(b, c)}(F) \arrow{r}\arrow{d}{[\Phi_{(b,c)}^{K,F}]} &\cdots
\\
	\cdots \arrow{r} &
	\Hf^\ast_{(a,b)}(K) \arrow{r}{[\iota_{a,a}^{b,c}]} &
	\Hf^\ast_{(a,c)}(K) \arrow{r}{[\pi_{a,b}^{c,c}]} &
	\Hf^\ast_{(b, c)}(K) \arrow{r} &\cdots
\end{tikzcd}
\end{equation}Suppose we are given another Hamiltonian $H\geq K$, then we have
the commutative diagram 

\[
\begin{tikzcd}[column sep = 40pt, row sep = 30pt]
	\Hf^\ast_{(a,b)}(F)\arrow{r}{[\Phi^{K,F}_{(a,b)}]}\arrow[bend right=30,swap]{rr}{[\Phi^{H,F}_{(a,b)}]}&
	\Hf^\ast_{(a,b)}(K)\arrow{r}{[\Phi^{H,K}_{(a,b)}]}&
	\Hf^\ast_{(a,b)}(H)
\end{tikzcd}
\]As opposed to the closed case, for completion of Liouville domains,
continuation maps do not necessarily yield isomorphisms. One case
in which they do is when both Hamiltonians have the same slope.
\begin{lem}[{\cite[Section 2.1]{We06}}]
\label{lem:Same_Slope}Let $F,K\in\mathcal{H}$ and suppose $\tau_{F}$
and $\tau_{K}$ are both contained in an open interval that does not
intersect $\Spec(\del D,\alpha)$. Then, 
\begin{itemize}
\item if $\tau_{F}=\tau_{K}$, $\Phi^{K,F}\comp\Phi^{F,K}=\id=\Phi^{F,K}\comp\Phi^{K,F}$
and thus $\Hf^{\ast}(F)\cong\Hf^{\ast}(K)$.
\item if $\tau_{F}<\tau_{K}$, $[\Phi^{K,F}]:\Hf^{\ast}(F)\to\Hf^{\ast}(K)$
is an isomorphism.
\end{itemize}
Under these isomorphisms, $1_{F}$ and $1_{K}$ are identified. 
\end{lem}

In action windows, we have the following isomorphisms. 
\begin{lem}[{\cite[ Proposition 1.1]{Vi99}}]
\label{lem:Homotopy_window} Let $H_{\bullet}$ be a monotone homotopy
between $H_{\pm}\in\mathcal{H}$ that is constant in the $s$-coordinate
for $|s|>S>0$. Suppose $a_{s},b_{s}:\RR\to\RR$ are functions which
are constant outside $[-S,S]$ and $a_{s},b_{s}\notin\Spec(H_{s})$
for all $s$. Then, 

\[
	\begin{tikzcd}
		{[\Phi^{H_-, H_+}]}:\Hf_{(a_{+},b_{+})}^{\ast}(H_{+})\arrow{r}{\cong}&\Hf_{(a_{-},b_{-})}^{\ast}(H_{-})
	\end{tikzcd}
\]for $a_{\pm}=\lim_{s\to\pm\infty}a_{s}$ and $b_{\pm}=\lim_{s\to\pm\infty}b_{s}$. 
\end{lem}

\subsection{Filtered Symplectic cohomology\label{subsec:Filtered-Symplectic-cohomology}}

Equip the set of admissible Hamiltonians $\mathcal{H}^{0}$ negative
on $D$ with the partial order 
\begin{align*}
H\preceq K\iff H(t,p)\leq K(t,p)\quad & \forall(t,p)\in S^{1}\times\hat{D}.
\end{align*}
Let $\{H_{i}\}_{i\in I}\subset\mathcal{H}^{0}$ be a cofinal sequence
with respect to $\preceq$. We define the \emph{symplectic cohomology
of $D$ }as the direct limit 
\begin{align*}
\SH_{(a,b)}^{\ast}(D)=\lim_{\substack{\longrightarrow\\
H_{i}
}
}\Hf_{(a,b)}^{\ast}(H_{i})
\end{align*}
taken with respect to the continuation maps 
\begin{align*}
[\Phi_{(a,b)}^{H_{j},H_{i}}] & :\Hf_{(a,b)}^{\ast}(H_{i})\to\Hf_{(a,b)}^{\ast}(H_{j})
\end{align*}
for $i<j$. We denote $\SH^{\ast}(D)=\SH_{(-\infty,+\infty)}^{\ast}(D)$.
The long exact sequence on Floer cohomology carries through the direct
limit and we also have a long exact sequence on symplectic cohomology
\[
\begin{tikzcd}[column sep = 25pt, row sep = 30pt]
	\SH^\ast_{(a,b)}(D) \arrow{dr}{[\iota_{a,a}^{b,c}]} 
&	
\\
&\SH^\ast_{(a,c)}(D) \arrow{dl}{[\pi_{a,b}^{c,c}]} 
\\
	\SH^\ast_{(b, c)}(D) \arrow{uu}{[+1]}
	&
\end{tikzcd}
\]

\noindent\textbf{The Viterbo map.--}Let $F\in\mathcal{H}$ and consider
$H\in\mathcal{H}^{0}$ with $\tau_{H}=\tau_{F}$. Then, by Lemma \ref{lem:Same_Slope},
we have $\Hf^{\ast}(F)\cong\Hf^{\ast}(H)$ and there exist, by definition
of symplectic cohomology, a map 
\begin{align}
j_{F}:\Hf^{\ast}(F)\cong\Hf^{\ast}(H)\to\SH^{\ast}(D)\label{eq:j}
\end{align}
sending each element of $\Hf^{\ast}(H)$ to its equivalence class.
Now, for $H\in\mathcal{H}^{0}$ with slope $\tau_{H}<T_{0}$ we can
define, by Lemma \ref{lem:Standard_cohomology} the map $v^{\ast}:\HH^{\ast}(D)\to\SH^{\ast}(D)$
first introduced in \cite{Vi99} by \[
	\begin{tikzcd}
	\label{Diag:viterbo}
		\HH^\ast(D)\arrow{r}{\cong}\arrow[bend right= 25, swap]{rr}{v^\ast}&
		\Hf^\ast(H)\arrow{r}{j_F}&
		\SH^\ast(D)
	\end{tikzcd}
\]This map induces a unit on symplectic cohomology. Recall that $1_{H}$
denotes the unit in $\Hf^{\ast}(H)$ (see Remark \ref{rem:Unit}). 
\begin{thm}[\cite{Ri13}]
\label{thm:The-ring-structure}The ring structure on $\Hf^{\ast}(H)$
induces a ring structure on $\SH^{\ast}(D)$. The unit on $\SH^{\ast}(D)$
is given by the image of the unit $e_{D}\in\HH^{\ast}(D)$ under the
map $v^{\ast}$. Moreover, 
\begin{align*}
v^{\ast}(e_{D})=[\iota_{-\infty,-\infty}^{\varepsilon,+\infty}](1_{H}).
\end{align*}
\end{thm}

\section{Spectral invariants and spectral norm\label{sec:Spectral-invariants}}

\subsection{Spectral invariants}

Denote by $\Ham_{c}(D,\dd\lambda)$ the group of compactly supported
Hamiltonian diffeomorphisms of $(D,\dd\lambda)$ and by $\Symp_{c}(D,\dd\lambda)$
the group of compactly supported symplectomorphisms of $(D,\dd\lambda)$.
The Hofer norm of a compactly supported Hamiltonian $H\in\mathcal{C}(D)$
is defined as 
\begin{align*}
\norm{H} & =\int_{0}^{1}\left(\sup_{p\in D}H(t,p)-\inf_{p\in D}H(t,p)\right)\,\dd t.
\end{align*}
Using the Hofer norm, we can define a bi-invariant metric \cite{Ho90,LaMc95}
on $\Ham_{c}(D,\dd\lambda)$ by 
\begin{align*}
d_{H}(\varphi,\psi)=d_{H}(\varphi\psi^{-1},\id),\quad d_{H}(\varphi,\id)=\inf\{\norm{H}\mid\varphi=\varphi_{H}\}.
\end{align*}

Recall that $\mathcal{C}(D)$ forms a group under the multiplication
\begin{align*}
H\#K(t,p)= & H(t,p)+K(t,(\varphi_{H}^{t})^{-1}(p))
\end{align*}
with the inverse of some $H\in\mathcal{C}(D)$ given by $\overline{H}(t,p)=-H(t,\varphi_{H}^{t}(p))$. 

From Lemma \ref{lem:Standard_cohomology} and by definition of $\Hf^{\ast}(H)$
for $H\in\mathcal{C}(D)$, we know that $\Hf^{\ast}(H)\cong\HH^{\ast}(D)$.
For $\beta\in\HH^{\ast}(D)$, we define, following \cite{Sc00}, the
\emph{spectral invariant of} $H$ \emph{relative to $\beta$ as}
\begin{align*}
c(\beta,H)=\inf\{\ell\in\RR\mid[\pi_{-\infty,\ell}^{+\infty,+\infty}]\comp[\iota_{-\infty,-\infty}^{\ell,+\infty}](\beta)=0\}.
\end{align*}
The following proposition gathers all the properties of spectral invariants
we need for the rest of the text. Proofs of these properties can be
found\footnote{Note that the signs for continuity and monoticity differ from \cite[ Section 5]{FrSc07}
because of differences in sign conventions.} in \cite[ Section 5]{FrSc07}.
\begin{prop}
\label{prop:Spectral_invariants}Let $\beta,\eta\in H^{\ast}(D)$
and let $H,K\in\mathcal{C}(D)$. Then, 
\begin{itemize}
\item \emph{{[}Continuity{]} }
\begin{align*}
\int_{0}^{1}\min_{x\in D}(K-H)\dd t\leq c(\beta,H)-c(\beta,K)\leq\int_{0}^{1}\max_{x\in D}(K-H)\dd t
\end{align*}
\item \emph{{[}Spectrality{]}} $c(\beta,H)\in\Spec(H)$.
\item \emph{{[}Triangle inequality{]}} $c(\beta\smile\eta,H\#K)\leq c(\beta,H)+c(\eta,K).$
\item \emph{{[}Monotonicity{]} If $H(t,x)\leq K(t,x)$ for all $(t,x)\in[0,1]\times D$,
then $c(\beta,H)\geq c(\beta,K)$. }
\end{itemize}
\end{prop}

\begin{rem}
The continuity property of Proposition \ref{prop:Spectral_invariants}
allows us to define spectral invariants of compactly supported continuous
Hamiltonians $H\in C_{c}^{0}([0,1]\times D)$. They satisfy continuity,
the triangle inequality and monotonicity. 
\end{rem}

\subsubsection{Additional properties of $c$}

The following lemma assures us that spectral invariants are well defined
on $\Ham_{c}(D,\dd\lambda)$. The proof relies on the spectrality
and the triangle inequality.
\begin{lem}
\label{lem:spectral_inv_only_depend}Let $H,K\in\mathcal{C}(D)$ such
that $\varphi_{H}=\varphi_{K}$ and let $\beta\in\HH^{\ast}(D)$.
Then, 
\begin{align*}
c(\beta,H)=c(\beta,K)
\end{align*}
\end{lem}

\begin{proof}
We have $\varphi_{H\#\overline{K}}=\varphi_{0}=\id$ and in that case
$\Spec(H\#\overline{K})=\{0\}$. Now, by spectrality of spectral invariants,
$c(\beta,H\#\overline{K})=0$. Thus, the triangle inequality yields
\begin{align*}
c(\beta,H)=c(\beta,H\#\overline{K}\#K)\leq c(\beta,H\#\overline{K})+c(\beta,K)=c(\beta,K).
\end{align*}
Repeating the same argument with $K\#\overline{H}$ instead of $H\#\overline{K}$,
we obtain $c(\beta,K)\leq c(\beta,H)$ which concludes the proof. 
\end{proof}
The spectral invariant with respect to the cohomological unit admits
an implicit definition which depends on the spectral invariants with
respect to all other cohomology classes in $\HH^{\ast}(D)$. This
follows directly from the triangle inequality. 
\begin{lem}
\label{lem:Def_implicite_c(1,H)}Let $H\in\mathcal{C}(D)$. Then,
\begin{align*}
c(1,H)=\max_{\beta\in\HH^{\ast}(D)}c(\beta,H).
\end{align*}
\end{lem}

\begin{proof}
Let $\beta\in H^{\ast}(D)$. By definition of the unit and the concatenation
of Hamiltonians, we have 
\begin{align*}
c(\beta,H)=c(\beta\smile1,H)=c(\beta\smile1,0\#H).
\end{align*}
Then, since $c(\beta,0)=0$, the triangle inequality guaranties that
\begin{align*}
c(\beta,H)=c(\beta\smile1,0\#H)\leq c(\beta,0)+c(1,H)=c(1,H).
\end{align*}
The choice of $\beta$ being arbitrary, this concludes the proof.
\end{proof}

\subsubsection{The symplectic contraction principle}

We conclude this section by recalling the symplectic contraction technique
introduced by Polterovich \cite[Section 5.4]{Po14}. This principle
allows one to describe the effect of the Liouville flow $\{\psi_{Y}^{\log r}\}_{0<r<1}$
on spectral invariants. 

First, we need to describe how the Liouville flow acts on the symplectic
form $\omega$ of $D$ and on compactly supported Hamiltonians on
$D$. Since $L_{Y}\omega=\omega$, we have that the Liouville flow
contracts the symplectic form :
\begin{align*}
\left(\psi_{Y}^{\log r}\right)^{\ast}\omega=r\omega
\end{align*}
Now, consider a Hamiltonian $H\in\mathcal{C}(D)$ supported in $U\subset D$.
For fixed $0<r<1$ define the Hamiltonian 
\begin{align}
H_{r}(t,x)=\begin{cases}
rH\left(t,\left(\psi_{Y}^{\log r}\right)^{-1}(x)\right) & \text{if }x\in\psi_{Y}^{\log r}(U),\\
0 & \text{if }x\notin\psi_{Y}^{\log r}(U).
\end{cases}\label{eq:Hr}
\end{align}
It then follows from the two previous equations that $\Spec(H_{r})=r\Spec(H)$.
This allows one to prove
\begin{lem}[\cite{Po14}]
\label{lem:Contraction_principle}Suppose $H\in\mathcal{C}(D)$ and
let $H_{r}\in\mathcal{C}(D)$ be as in Equation \ref{eq:Hr}. Then,
\begin{align*}
c(1,H_{r})=rc(1,H).
\end{align*}
\end{lem}

\subsection{Spectral norm}

We define the \emph{spectral norm $\gamma(H)$ of }$H\in\mathcal{C}(D)$
as 
\begin{align*}
\gamma(H)=c(1,H)+c(1,\overline{H}).
\end{align*}
For $\varphi\in\Ham_{c}(D,\dd\lambda)$ such that $\varphi=\varphi_{H}$,
define
\begin{align*}
\gamma(\varphi)=\gamma(H)
\end{align*}
In virtue of Lemma \ref{lem:spectral_inv_only_depend}, this is well
defined. 

From \cite[ Section 7]{FrSc07}, we have the following theorem which
justifies calling $\gamma$ a norm.
\begin{thm}
Let $\varphi,\psi\in\Ham_{c}(D,\dd\lambda)$ and let $\chi\in\Symp_{c}(D,\dd\lambda)$.
Then,
\begin{itemize}
\item \emph{{[}Non-degeneracy{]}} $\gamma(\id)=0$ and $\gamma(\varphi)>0$
if $\gamma\neq\id$,
\item \emph{{[}Triangle inequality{]}} $\gamma(\varphi\psi)\leq\gamma(\varphi)+\gamma(\psi)$,
\item \emph{{[}Symplectic invariance{]}} $\gamma(\chi\comp\varphi\comp\chi^{-1})=\gamma(\varphi)$,
\item \emph{{[}Symmetry{]}} $\gamma(\varphi)=\gamma(\varphi^{-1})$,
\item \emph{{[}Hofer bound{]}} $\gamma(\varphi)\leq d_{H}(\varphi,\id)$. 
\end{itemize}
\end{thm}

\section{Cohomological barricades on Liouville domains}

In \cite{GaTa20} Ganor and Tanny introduced a particular perturbation
of Hamiltonians compactly supported inside \emph{contact incompressible
boundary }domains (CIB) of closed aspherical symplectic manifolds.
For instance, if $U\subset M$ is an incompressible open set which
is a Liouville domain, then $U$ is a CIB. In Floer homology, the
aforementioned Hamiltonian perturbation, which is called a barricade,
prohibits the existence of Floer trajectories exiting and entering
the CIB. We consider barricades in the particular case of Liouville
domains and adapt them to Floer cohomology. 
\begin{defn}
\label{def:Baricade}Let $r_{0}>1$ and $0<\varepsilon<r_{0}-1$.
Define $B_{r_{0},\varepsilon}=D^{r_{0}-\varepsilon}\setminus D$ where,
for $\rho>0$, $D^{\rho}=\Psi_{Y}^{\log\rho}(D)$. Suppose $(F_{\bullet},J)$
is a pair of a monotone homotopy $F_{\bullet}$ from $F_{+}$ to $F_{-}$
and an admissible almost complex structure $J$. We say that $(F_{\bullet},J)$
admits a barricade on $B_{r_{0},\varepsilon}$ if for every $x_{\pm}\in\mathcal{\mathcal{P}}(F_{\pm})$
and every Floer trajectory $u:\RR\times S^{1}\to\hat{D}$ connecting
$x_{\pm}$, we have, for $\Db:=D^{r_{0}-\varepsilon}=D\cup B_{r_{0},\varepsilon}$
\begin{enumerate}
\item If $x_{-}\in D$, then $\im(u)\subset D$,
\item If $x_{+}\in D_{b}$, then $\im(u)\subset\Db$. 
\end{enumerate}
\end{defn}

\begin{figure}[H]
\begin{centering}
\includegraphics{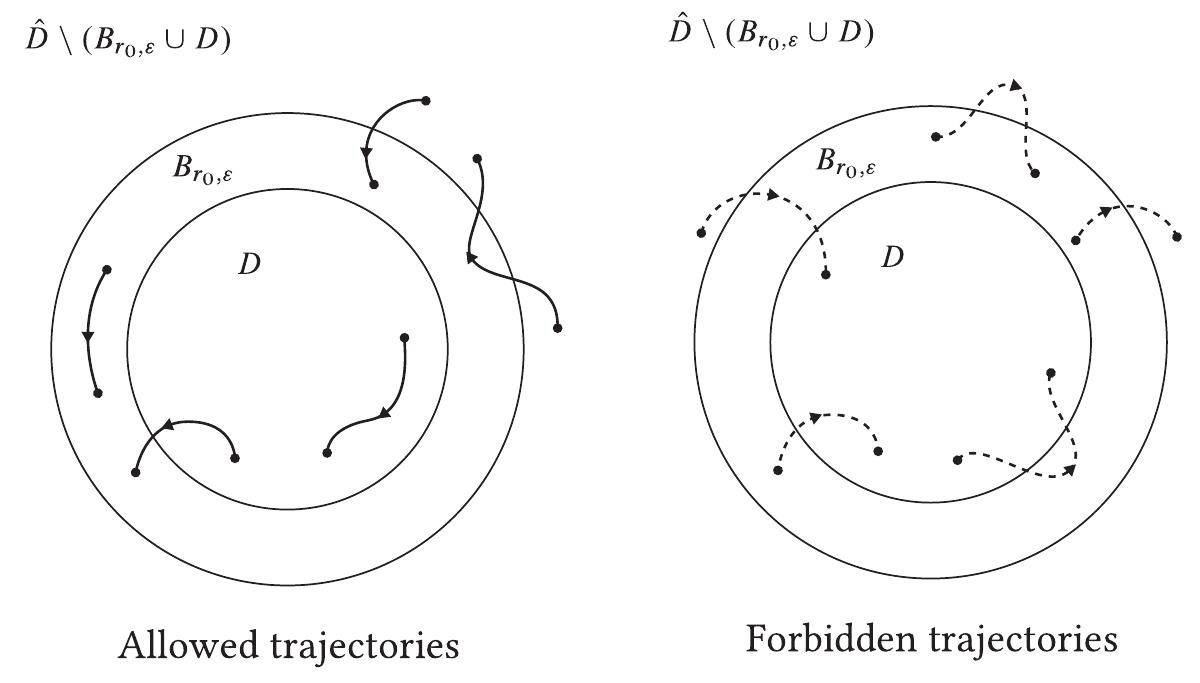}
\par\end{centering}
\caption{Floer cylinders in a barricade.}
\end{figure}

\begin{rem}
In the language of \cite{GaTa20}, a barricade on $B_{r_{0},\varepsilon}$
as described above would be called a barricade in $D^{r_{0}-\varepsilon}$
around $D$. 
\end{rem}

\subsection{How to construct barricades}

To construct barricades, we consider the following class of pairs. 
\begin{defn}
Let $r_{0}>1$, $\sigma\in(0,+\infty)\setminus\Spec(\partial D,\lambda)$
and $0<\varepsilon<r_{0}-1$. The pair $(F_{\bullet},J)$ admits a
cylindrical bump of slope $\sigma$ on $B_{r_{0},\varepsilon}$ if 
\begin{itemize}
\item $F=0$ on $\partial B_{r_{0},\varepsilon}\times S^{1}\times\RR$,
\item $JY=R_{\alpha}$, for $Y$ the Liouville vector field on $D$, on
a neighborhood of $\partial B_{r_{0},\varepsilon}$, i.e. $J$ is
cylindrical near $\partial B_{r_{0},\varepsilon}=\partial D\sqcup(\{r_{0}-\varepsilon\}\times\partial D)$.
\item $\nabla_{J}F=\sigma Y$ near $(\{1\}\times\partial D)\times S^{1}\times\RR$
and $\nabla_{J}F=-\sigma Y$ near $(\{r_{0}-\varepsilon\}\times\partial D)\times S^{1}\times\RR$.
Here, $\nabla_{J}$ denotes the gradient induced by the metric $g_{J}$.
\item All 1-periodic orbits of $F_{\pm}$ contained in $B_{r_{0},\varepsilon}$
are critical points with values in the interval $(-\sigma,\sigma)$.
(In particular, $\sigma<T_{0}$.) 
\end{itemize}
A cohomological adaptation of Lemma 3.3 in \cite{GaTa20} yields the
following action estimates for pairs with cylindrical bumps. 
\end{defn}

\begin{lem}
\label{lem:act_estimates}Suppose that $(F,J)$ admits a cylindrical
bump of slope $\sigma$ on $B_{r_{0},\varepsilon}$. For every finite
energy solution $u$ connecting $x_{\pm}\in\mathcal{P}(F_{\pm})$,
then 
\begin{itemize}
\item $x_{-}\subset D$ and $x_{+}\subset\hat{D}\setminus D$ $\implies$
$\mathcal{A}_{F_{+}}(x_{+})>\sigma$,
\item $x_{+}\subset D$ and $x_{-}\subset\hat{D}\setminus D$ $\implies$
$\mathcal{A}_{F_{-}}(x_{-})>\sigma$, 
\item $x_{-}\subset\Db$ and $x_{+}\subset\hat{D}\setminus\Db$ $\implies$
$\mathcal{A}_{F_{+}}(x_{+})<-\sigma$,
\item $x_{+}\subset\Db$ and $x_{-}\subset\hat{D}\setminus\Db$ $\implies$
$\mathcal{A}_{F_{-}}(x_{-})<-\sigma$. 
\end{itemize}
\end{lem}

Lemma \ref{lem:act_estimates} and the maximum principle are all we
need to prove that every pair with a cylindrical bump admits a barricade.
More precisely, we have the 
\begin{prop}
\label{prop:Cylindrical_barricade}Let $(F,J)$ be a pair with a cylindrical
bump of slope $\sigma$ on $B_{r_{0},\varepsilon}$. Then, $(F,J)$
admits a barricade on $B_{r_{0},\varepsilon}$.
\end{prop}

\begin{proof}
Suppose $u:\RR\times S^{1}\to\hat{D}$ is a Floer trajectory between
$x_{\pm}\in\mathcal{P}(F_{\pm})$. We only need to study the case
where $x_{-}\in D$ and the case where $x_{+}\in\Db$. 

Suppose that $x_{-}\in D$. We first establish that $x_{+}$ must
lie inside $D$. Indeed, if $x_{+}\in\hat{D}\setminus D$, Lemma \ref{lem:act_estimates}
assures us that $\mathcal{A}_{F_{+}}(x_{+})>\sigma$ which contradicts
the fact that orbits on $\hat{D}\setminus D$ must have action in
the interval $(-\sigma,\sigma)$ by the construction of the cylindrical
bump. Therefore, $x_{+}\in D$ as desired. Now, since $x_{\pm}\in D$,
the maximum principle guarantees that $\im u\subset D$. 

To finish the proof, we look at the case where $x_{+}\in\Db$. Similarly
to the previous case, we prove that $x_{-}$ also lies inside $\Db$.
If $x_{-}\in\hat{D}\setminus\Db$, Lemma \ref{lem:act_estimates}
imposes $\mathcal{A}_{F_{-}}(x_{-})<-\sigma$, which is again impossible
by construction of the cylindrical bump. Therefore, $x_{-}\in\Db$
and the maximum principle implies $\im u\subset\Db$. 
\end{proof}
Given a pair $(F,J)$ and $\sigma>0$ small, we can add to $F$ a
$\mathcal{C}^{\infty}$-small radial bump function $\chi$ with support
inside $B_{r_{0},\varepsilon}$ such that $(F+\chi,J)$ has a cylindrical
bump of slope $\sigma$ on $B_{r_{0},\varepsilon}$. By Proposition
\ref{prop:Cylindrical_barricade}, the perturbed pair will also admit
a barricade on $B_{r_{0},\varepsilon}$. A second perturbation of
the Hamiltonian term at its ends, under which the barricade survives,
allows us to achieve Floer regularity for the pair. This procedure
is carried out carefully in \cite[section 9]{GaTa20} and yields the
following. 
\begin{thm}[\cite{GaTa20}]
Let $F_{\bullet}$ be a monotone homotopy. Then, there exists a $\mathcal{{C}}^{\infty}$-small
perturbation $f_{\bullet}$ of $F_{\bullet}$ and an almost complex
structure $J$ such that the pairs $(f_{\bullet},J)$ and $(f_{\pm},J)$
are Floer-regular and have a barricade on $B_{r_{0},\varepsilon}$.
\end{thm}

\subsection{Decomposition of the Floer cochain complex}

\label{subsec:Decomp_CF}

Let us investigate what structure barricades impose on the Floer co-chain
complex. Let $H\in\mathcal{H}_{r_{0}}$ and suppose the pair $(H,J)$
admits a barricade on $B_{r_{0},\varepsilon}$. For an open subset
$U\subset\hat{D}$, denote by $\CO^{\ast}(U,H)$ the set of $1$-periodic
orbits of $H$ in $U$. By definition of the differential $\partial$
on Floer cohomology, $\CO^{\ast}(\Db,H)$ is closed under $\partial$
and it therefore forms a sub-complex of $\CF^{\ast}(H)$. Moreover,
for $\Dc=\hat{D}\setminus\Db$, we also have that 
\begin{align*}
\CO^{\ast}(\Dc,H) & =\frac{\CF^{\ast}(H)}{\CO^{\ast}(\Db,H)}
\end{align*}
is a well defined cochain complex.

\subsubsection{Continuation maps}

Let $(F_{\bullet},J)$ be a pair that admits a barricade on $B_{r_{0},\varepsilon}$
where $F_{\bullet}$ is a monotone homotopy from $F_{+}$ to $F_{-}$.
Then, since the continuation map $\Phi_{F_{\bullet}}:\CF^{\ast}(F_{+})\to\CF^{\ast}(F_{-})$
counts Floer trajectories of $F$ connecting 1-periodic orbits of
$F_{+}$ to 1-periodic orbits of $F_{-}$, it restricts, due to the
barricade, to a chain map

\begin{align*}
\Phib_{F}:\CO^{\ast}(\Db,F_{+})\to\CO^{\ast}(\Db,F_{-}) & .
\end{align*}
Moreover, in virtue of Lemma \ref{lem:map_quotient} below, $\Phi_{F}$
projects to a chain map 
\begin{align*}
\Phic_{F} & :\CO^{\ast}(\Dc,F_{+})\to\CO^{\ast}(\Dc,F_{-})
\end{align*}
 such that the following diagram commutes 

\[
\begin{tikzcd}[column sep = 30pt, row sep = 30pt]
	\CF^\ast(F_+)
		\arrow{r}{\Phi_{F_\bullet}}
		\arrow{d}[swap]{\pi_+^{\text{b}}}
	&
	\CF^\ast(F_-)
		\arrow{d}{\pi_-^{\text{b}}}
	\\
	\CO^\ast(\Dc, F_+) \arrow{r}[swap]{\Phic_{F_\bullet}}
	&
	\CO^\ast(\Dc, F_-)
\end{tikzcd}
\]for $\pi_{+}^{\text{b}}$ and $\pi_{-}^{\text{b}}$ the canonical
projections. 

\subsubsection{Chain homotopies}

For $F_{\pm}\in\mathcal{H}_{r_{0}}$ , consider the linear homotopy
\begin{align*}
F_{s} & =F_{-}+\beta(s)(F_{+}-F_{-})
\end{align*}
where $\beta:\RR\to[0,1]$ is a smooth function such that $\beta(s)=0$
for $s\leq-1$, $\beta(s)=1$ for $s\geq1$ and $\beta'(s)>0$ for
all $s\in(-1,1)$. Denote by $\overline{F}_{\bullet}$ the inverse
homotopy defined by $\overline{F}_{s}=F_{-s}$. For $\rho>1$ large,
we define the concatenation $F\#\overline{F}_{\bullet}$ as 
\begin{align*}
(F\#\overline{F})_{s} & =\begin{cases}
F_{s+\rho} & \text{for }s\leq0\\
\overline{F}_{s-\rho} & \text{for }s\geq0
\end{cases}.
\end{align*}
Using the definition of $F_{\bullet}$ and $\overline{F}_{\bullet}$,
we can simply write 
\begin{align*}
(F\#\overline{F})_{s} & =F_{-}+\beta_{\rho}(s)(F_{+}-F_{-})
\end{align*}
for $\beta_{\rho}(s)=\beta(-|s|+\rho)$. The homotopy $F\#\overline{F}_{\bullet}$
generates the composition of continuation homomorphisms $\Phi_{F}\comp\Phi_{\overline{F}}:\CF^{\ast}(F_{-})\to\CF^{\ast}(F_{-})$
which is chain homotopic to the identity on $\CF^{\ast}(F_{-})$,
\begin{align*}
\Phi_{F_{\bullet}}\comp\Phi_{\overline{F}_{\bullet}}-\id_{-} & =\partial_{-}\comp\Psi_{-}-\Psi_{-}\comp\partial_{-}
\end{align*}
for $\Psi_{-}:\CF^{\ast}(F_{-})\to\CF^{\ast-1}(F_{-})$ and $\partial_{-}$
the differential on $\CF^{\ast}(F_{-})$. The chain homotopy $\Psi_{-}$
is built by counting Floer solutions of the homotopy $\ens{\Gamma^{\kappa}}_{\kappa\in[0,1]}$
between $F\#\overline{F}_{\bullet}$ and the constant homotopy $F_{-}$
which is defined by 
\begin{align*}
\Gamma_{s}^{\kappa}=F_{-}+\kappa\beta_{\rho}(s)(F_{+}-F_{-}).
\end{align*}
For $x\in\mathcal{P}(F_{-})$ and $y\in\mathcal{P}(F_{+})$, define
\begin{align*}
\mathcal{M}^{\Gamma}(x,y)=\ens{(\kappa,u)\mid\kappa\in[0,1],\ u\in\mathcal{M}(x,y;\Gamma_{\bullet}^{k})}.
\end{align*}
We can perturb $\Gamma$ with a $C^{\infty}$-small function in order
to make it regular \cite[ Chapter 11]{AuDa14}. Now, if $F_{-}$ and
$F_{+}$ admit barricades on $B_{r_{0},\varepsilon}$, solutions to
the parametric Floer equation for $\Gamma^{\kappa}$ also admit barricades
on $B_{r_{0},\varepsilon}$. The same holds with its regular perturbation. 
\begin{lem}
\label{lem:barricade_homotopy2}Let $F_{-},F_{+}\in\mathcal{H}_{r_{0}}$
and suppose they both admit a barricade on $B_{r_{0},\varepsilon}$.
Then, for any $C^{\infty}$-small perturbation $\Gamma'$ of $\Gamma$
which satisfies $\mathcal{P}(F'_{\pm})=\mathcal{P}(F_{\pm})$, Floer
trajectories in $\mathcal{M}^{\Gamma^{\kappa}}$ follow the rules
of the barricade on $B_{r_{0},\varepsilon}$. 
\end{lem}

\begin{proof}
The proof follows the same ideas as the proof of Proposition 9.21
in \cite{GaTa20}. By Gromov compactness, any sequence $(\kappa_{n},u_{n})\in\mathcal{M}^{\Gamma}(x_{-},y_{+})$
of solutions to the parametric Floer equation converges, up to taking
a subsequence, to a broken trajectory $(\kappa,\bar{v})$ where $\bar{v}=(v_{1},\ldots,v_{k},w,v_{1}',\ldots,v_{\ell}')$
connects two orbits $x_{\pm}\in\mathcal{P}(F_{\pm})$ . The fact that
$F_{\pm}$ both admit a barricade on $B_{r_{0},\varepsilon}$ assures
us that 
\begin{itemize}
\item $x_{-}\in D$ $\implies$ $\bar{v}\subset D$
\item $x_{+}\in D$ $\implies$ $\bar{v}\subset\Db$.
\end{itemize}
Now, consider a sequence of regular homotopies $\ens{\Gamma_{n}}_{n}$
with ends $\lim_{s\to\pm\infty}\Gamma_{s,n}=F_{n\pm}$ converging
to $\Gamma$ such that $\mathcal{P}(F_{n\pm})=\mathcal{P}(F_{\pm})$
for all $n$. Then, the above two implications regarding broken trajectories
imply that every trajectory $(\kappa_{n},u_{n}')\in\mathcal{M}^{\Gamma}(x_{-},x_{+})$,
for $x_{\pm}\in\mathcal{P}(F_{\pm})$, obeys to the rules of the barricade. 
\end{proof}
Thus, $\Psi_{-}$ restricts to a map $\Psi_{-}^{b}:\CO^{\ast}(\Db,F_{-})\to\CO^{\ast-1}(\Db,F_{-})$
and by Lemma \ref{lem:chain_homotopy} below, we can define its projection
$\Psi_{-}^{c}:\CO^{\ast}(\Dc,F_{-})\to\CO^{\ast-1}(\Dc,F_{-})$.

\emph{Technical lemmas.-- }When adapting computations from homology
to cohomology, we often have to rely on quotient complexes instead
of sub-complexes. Here are a few simple results from homological algebra
which will be useful in that regard. Let $(A,d_{A})$ and $(C,d_{C})$
be cochain complexes and let $B\subset A$ and $D\subset C$ be sub-complexes.
\begin{lem}
\label{lem:map_quotient}Suppose $f:(A,B)\to(C,D)$ is a chain map.
Then, there exists a unique chain map $\bar{f}:A/B\to C/D$ such that
the following diagram commutes

\[
\begin{tikzcd}[column sep = 30pt, row sep = 30pt]
	A
		\arrow{r}{f}
		\arrow{d}[swap]{\pi_B}
	&
	C
		\arrow{d}{\pi_D}
	\\
	A/B \arrow{r}[swap]{\bar{f}}
	&
	C/D
\end{tikzcd}
\]for $\pi_{B}$ and $\pi_{D}$ the canonical projections. It follows
that, on cohomology, we have the following commutative diagram. 

\[
\begin{tikzcd}[column sep = 30pt, row sep = 30pt]
	\HH^\ast (A)
		\arrow{r}{[f]}
		\arrow{d}[swap]{[\pi_B]}
	&
	\HH^\ast(C)
		\arrow{d}{[\pi_D]}
	\\
	\HH^\ast(A/B) \arrow{r}[swap]{[\bar{f}]}
	&
	\HH^\ast(C/D)
\end{tikzcd}
\]
\end{lem}

\begin{proof}
Define, for all $x\in A$, 
\begin{align*}
\bar{f}(\pi_{B}(x)) & =\pi_{D}(f(x)).
\end{align*}
We first need to show that $\bar{f}$ is well defined. Suppose $x'=x+b$
for $x\in A$ and $b\in B$. Then, since $f$ restricts to a map from
$B$ to $D$, there exists $d\in D$ such that $f(b)=d$ and we have
\begin{align*}
\bar{f}(\pi_{B}(x')) & =\pi_{D}(f(x+b))=\pi_{D}(f(x)+d)=\pi_{D}(f(x)).
\end{align*}
Thus, $\bar{f}$ is well defined. 

To prove uniqueness, we simply use the definition of $\bar{f}$. Suppose
we have another map $\bar{g}:A/B\to C/D$ which makes the above diagram
commute as well. Then, for all $x\in A$,
\begin{align*}
\bar{f}(\pi_{B}(x))-\bar{g}(\pi_{B}(x)) & =\pi_{D}(f(x))-\pi_{D}(f(x))=0.
\end{align*}
\end{proof}
\begin{lem}
\label{lem:chain_homotopy}Suppose $f:(A,B)\to(C,D)$ and $g:(C,D)\to(A,B)$
are chain maps such that $f\comp g$ is chain homotopic to the identity
\begin{align*}
f\comp g-\id_{C} & =d_{C}\comp\psi-\psi\comp d_{C}
\end{align*}
where the chain homotopy is a map $\psi:(C,D)\to(C,D)$. Then, $\bar{f}\comp\bar{g}:C/D\to C/D$
is also chain homotopic to the identity. 
\end{lem}

\begin{proof}
Since the chain homotopy $\psi:(C,D)\to(C,D)$ is a chain map of pairs,
Lemma \ref{lem:map_quotient} allows us to define its projection $\bar{\psi}:C/D\to C/D$.
Thus, for all $y\in C$,
\begin{align*}
\bar{f}\comp\bar{g}(\pi_{D}(y))-\id_{C/D}(\pi_{D}(y)) & =\bar{f}\comp\pi_{B}(g(y))-\pi_{D}(\id_{C}(y))\\
 & =\pi_{D}(f\comp g(y))-\pi_{D}(\id_{C}(y))\\
 & =\pi_{D}((d_{C}\comp\psi-\psi\comp d_{C})(y))\\
 & =(d_{C/D}\comp\pi_{D}\comp\psi-\pi_{D}\comp\psi\comp d_{C})(y)\\
 & =d_{C/D}\comp\bar{\psi}(\pi_{D}(y))-\bar{\psi}\comp d_{C/D}(\pi_{D}(y))
\end{align*}
which proves that $\bar{f}\comp\bar{g}$ is chain homotopic to the
identity on $C/D$ since any $z\in C/D$ is of the form $z=\pi_{D}(y)$.
\end{proof}

\section{Proofs of main results}

\subsection{Proof of Theorem \ref{thm:finite_diameter}}

\noindent Fix $A\in(0,\infty)\setminus\Spec(\partial D,\lambda)$.
The idea of the proof is to construct a special admissible Hamiltonian
for which $c(1,\cdot)$ is bounded from below by $A-\varepsilon$
for $\varepsilon$ a small constant which depends on $A$. This construction
is inspired by \cite[Proposition 2.5]{CiFrOa10}. Then, we use the
fact that $c(1,\cdot)\geq0$ to conclude. 

\subsubsection{Construction of the Hamiltonian}

For any $\delta\in(0,1)$ and $\sigma\in(0,T_{0})$, we define the
Hamiltonian $H_{\delta,A}\in\mathcal{H}_{r_{0}}$ as follows:
\begin{itemize}
\item $H_{\delta,A}$ is the constant function $A(\delta-1)$ on $D^{\delta}$,
\item $H_{\delta,A}(r,x)=A(r-1)$ on $D\setminus D^{\delta}$,
\item $H_{\delta,A}(r,x)=0$ on $D^{r_{0}}\setminus D$
\item $H_{\delta,A}(r,x)=\sigma(r-r_{0})$ on $\hat{D}\setminus D^{r_{0}}$.
\end{itemize}
\begin{figure}[h]
\begin{centering}
\includegraphics{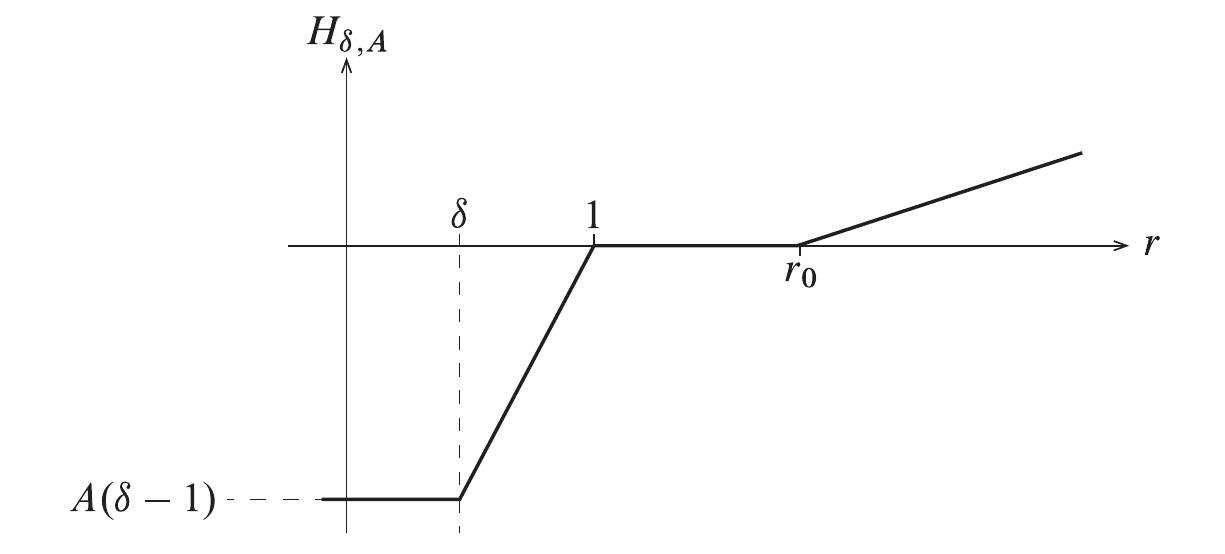}
\par\end{centering}
\caption{Radial portion of the Hamiltonian $H_{\delta,A}$.}
\end{figure}
We add a small perturbation to $H_{\delta,A}$ so that it lies in
$\mathcal{H}_{r_{0}}$. Denote by $h_{\delta,A}$ the restriction
of $H_{\delta,A}$ to $\hat{D}\setminus D$. If $\gamma$ is a 1-periodic
orbit of $h_{\delta,A}$ inside the level set $\ens{r}\times\partial D,$
its action can be written as 

\[
\Ac_{H_{\delta,A}}(\gamma)=\Ac_{H_{\delta,A}}(r)=rh'_{\delta,A}(r)-h_{\delta,A}(r).
\]
The 1-periodic orbits of $H_{\delta,A}$ can be classified in three
different categories. Recall that $\eta_{A}$ denotes the distance
between $A$ and $\Spec(\partial D,\alpha)$. 

\begin{itemize}
	\item[(I)] Critical points in $D^\delta$ with action close to $r_{\text{I}}=(1-\delta)A$
	\item[(II)] Non-constant 1-periodic orbits near $\ens{\delta}\times \del D$ with action in a small neighborhood of the interval 
	\[
		I_{\text{II}}=[\delta T_0 + (1 - \delta)A,\, A-\delta\eta_A].
	\]
	\item[(III)] Non-constant 1-periodic orbits near $\ens{1}\times \del D$ with action in a small neighborhood of the interval
	\[
		I_{\text{III}} = [T_0,\,A-\eta_A].
	\]
	\item[(IV)] Critical points in $D^{r_0}\setminus D$ with action close to $r_{\text{IV}} = 0$.
\end{itemize}Note that there are no non-constant 1-periodic orbits near $\ens{r_{0}}\times\partial D$,
since the slope of the Hamiltonian there ranges from $0$ to $\sigma$
which is less than $T_{0}$ by assumption. 

We now want to construct a Floer complex $\CO_{\text{{I,II}}}^{\ast}$
which will contain the orbits of type (I) and (II) and another complex
$\CO_{\text{{III,IV}}}^{\ast}$ containing orbits of type (III) and
$(\text{IV})$. To that end, pick $0<\delta<1$ small enough so that
$\delta A<\eta_{A}$. Now choose $\varepsilon>0$ such that 
\[
\delta A<\varepsilon<\eta_{A}.
\]
Then, we have the following inequalities :
\[
r_{\text{IV}}<I_{\text{III}}<A-\varepsilon<r_{\text{I}}<I{}_{\text{II}}.
\]

\begin{figure}[h]
\begin{centering}
\includegraphics{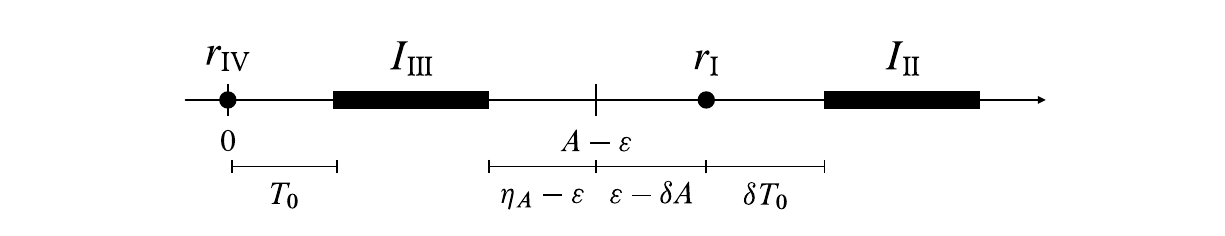}
\par\end{centering}
\caption{Distances that separate the action windows under consideration.}
\label{action_windows}
\end{figure}
\noindent As shown in Figure \ref{action_windows}, $r_{\text{\text{I}}}$,
$I_{\text{II}}$, $I_{\text{III}}$ and $r_{\text{IV}}$ are all separated
by distances which depend only on $T_{0}$, $A$, $\eta_{A}$, $\delta$
and $\varepsilon$. Thus, we can choose the perturbation we add to
$H_{\delta,A}$ to be small enough so that, in terms of action, we
have 

\[
(\text{IV})<(\text{{III}})<A-\varepsilon<(\text{{I}})<(\text{{II}}).
\]
Therefore, since the Floer differential decreases the action, we can
define the Floer co-chain complexes as
\[
\CO_{\text{{III,IV}}}^{\ast}=\CF_{<A-\varepsilon}^{\ast}(H_{\delta,A}),\quad\CO_{\text{{I,II}}}^{\ast}=\frac{{\CF^{\ast}(H_{\delta,A})}}{\CO_{\text{{III,IV}}}^{\ast}}=\CF_{(A-\varepsilon,\infty)}^{\ast}(H_{\delta,A})
\]
and they yield the Floer cohomology groups 
\begin{align*}
\HH^{\ast}(\CO_{\text{III,IV}}^{\ast})=\Hf_{(-\infty,A-\varepsilon)}^{\ast}(H_{\delta,A}),\quad\HH^{\ast}(\CO_{\text{I,II}}^{\ast})=\Hf_{(A-\varepsilon,\infty)}^{\ast}(H_{\delta,A}).
\end{align*}

A quick look at the action windows under consideration informs us
that the above complexes fit into the following short exact sequence 

\[
\begin{tikzcd}[column sep = 40pt, row sep = 30pt]
	0\arrow{r}
	&
	\CO_{\text{III,IV}}^\ast
	\arrow{r}{\iota_{-\infty,-\infty}^{A-\varepsilon,+\infty}}
	&
	\CF^\ast(H_{\delta, A})
	\arrow{r}{\pi_{-\infty, A-\varepsilon}^{+\infty,+\infty}}
	&
	\CO_{\text{I, II}}^\ast
	\arrow{r}
	&
	0
\end{tikzcd}
\]which in turn yields an exact triangle in cohomology

\[
\begin{tikzcd}
\HH^\ast(\CO_{\text{III,IV}}^\ast)
	\arrow{rr}{[\iota_{-\infty,-\infty}^{A-\varepsilon,+\infty}]}
&&
\Hf^\ast(H_{\delta,A})
	\arrow{dl}{[\pi_{-\infty, A-\varepsilon}^{+\infty,+\infty}]}
\\
&
\HH^\ast(\CO_{\text{I,II}}^\ast)
	\arrow{ul}{{[+1]}}
&
\end{tikzcd}
\]

\subsubsection{Factoring a map to $\SH^{\ast}(D)$}

We now build maps $\Psi$ and $\Psi_{\text{{I,II}}}$ such that the
diagram

\begin{equation}
\label{triangle}
\begin{tikzcd}[row sep=40pt, column sep=50pt]
\Hf^\ast(H_{\delta,A})
	\arrow{r}{[\pi_{-\infty, A-\varepsilon}^{\infty,\infty}]}
	\arrow{dr}[swap]{\Psi}
&
\HH^\ast(\CO^\ast_{\text{I,II}})
	\arrow{d}{\Psi_{\text{I,II}}}
\\
&
\SH^\ast(D)
\end{tikzcd}
\end{equation}commutes. We need to construct $\Psi$ so that it coincides with the
map $j_{H_{\delta,A}}:\Hf^{\ast}(H_{\delta,A})\to\SH^{\ast}(D)$ (see
Equation \ref{eq:j}). In virtue of Theorem \ref{thm:The-ring-structure},
this assures us that $\Psi$ is a map of unital algebras.

First, we construct $\Psi_{\text{{I,II}}}$ in three steps.

STEP 1. $[\Phi_{1}]:\HH^{\ast}(\CO_{\text{{I,II}}}^{\ast})\cong\Hf_{(\delta A-\varepsilon,\infty)}^{\ast}(H_{\delta,A}+A(1-\delta)).$
This isomorphism follows from a simple shift of $A(1-\delta)$ in
the Hamiltonian term which translates to a shift of $A(\delta-1)$
in action (see Figure \ref{Fig:homotopy_shift}). In what follows,
we denote $\hat{H}_{\delta,A}:=H_{\delta,A}+A(1-\delta)$.

\begin{figure}[H]
\begin{centering}
\includegraphics{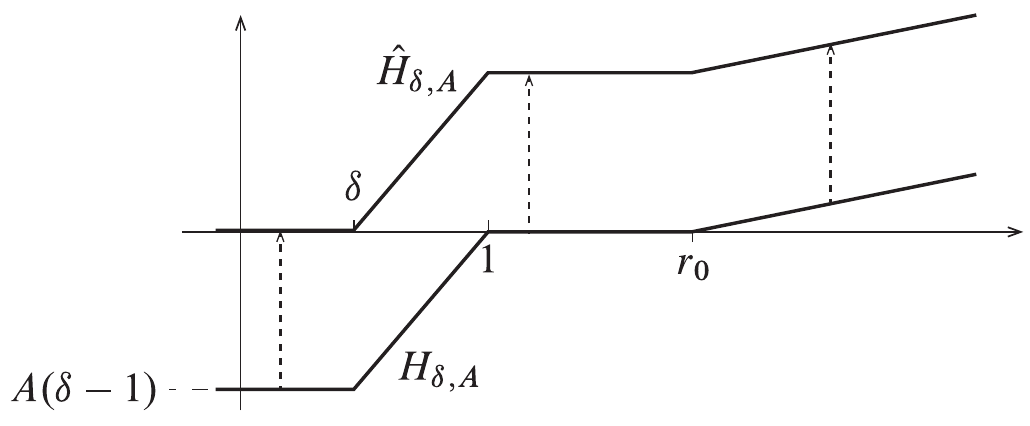}
\par\end{centering}
\caption{Homotopy from $H_{\delta,A}$ to $\hat{H}_{\delta,A}.$}
\label{Fig:homotopy_shift}
\end{figure}

For the next steps, we need to define another special family of Hamiltonians.
Given $r_{1}\in(0,+\infty)$ and $\tau\in(0,\infty)\setminus\Spec(\del D,\lambda)$,
define the Hamiltonian $K_{r_{1},\tau}$ as follows (see Figure \ref{Fig:K_Hamiltonian}).
\begin{itemize}
\item $K_{r_{1},\tau}$ is a $C^{2}$-small perturbation of the constant
zero function on $D^{r_{1}}$,
\item $K_{r_{1},\tau}(x,r)=\tau(r-r_{1})$ on $\hat{D}\setminus D^{r_{1}}$.
\end{itemize}
The 1-periodic orbits of a suitable perturbation of $K_{r_{1},\tau}$
fall in two categories. 

\begin{itemize}
	\item[(I')] Critical points in $D^{r_1}$ with action near zero,
	\item[(II')] Non-constant 1-periodic orbits near $\ens{r_1}\times \del D$ with action in a small neighborhood of the interval 
	\[
		[r_1T_0, r_1\tau- r_1\eta_\tau].
	\]
\end{itemize}By the same argument used for $H_{\delta,A}$, the action windows
$(\text{I}')$ and $(\text{II}')$ are separated if we choose a small
enough perturbation. 

\begin{figure}[h]
\begin{centering}
\includegraphics{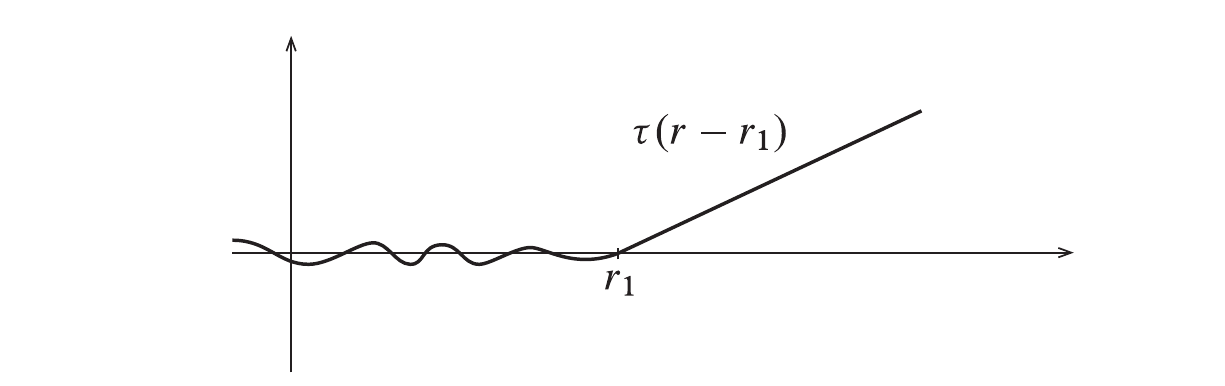}
\par\end{centering}
\caption{Radial portion of the Hamiltonian $K_{r_{0},\tau}$.}

\label{Fig:K_Hamiltonian}
\end{figure}

STEP 2. $[\Phi_{2}]:\Hf_{(\delta A-\varepsilon,\infty)}^{\ast}(\hat{H}_{\delta,A})\cong\Hf_{(\delta A-\varepsilon,\infty)}^{\ast}(K_{\delta,A}).$
Consider the homotopy
\[
F_{s}=(1-\beta(s))K_{\delta,A}+\beta(s)\hat{H}_{\delta,A},
\]
where $\beta:\RR\to[0,1]$ is a smooth function such that $\beta(s)=0$
for $s\leq-1$, $\beta(s)=1$ for $s\geq1$ and $\beta'(s)>0$ for
all $s\in(-1,1)$ (see Figure \ref{Fig:homotopy_lift}). Denote by 

\[
\begin{tikzcd}
\Phi_{F_\bullet}:\CF^\ast(\hat{H}_{\delta,A})
	\arrow{r}
&
\CF^\ast(K_{\delta,A})
\end{tikzcd}
\]the continuation map generated by $F_{\bullet}$. 

\begin{figure}[h]

\begin{centering}
\includegraphics{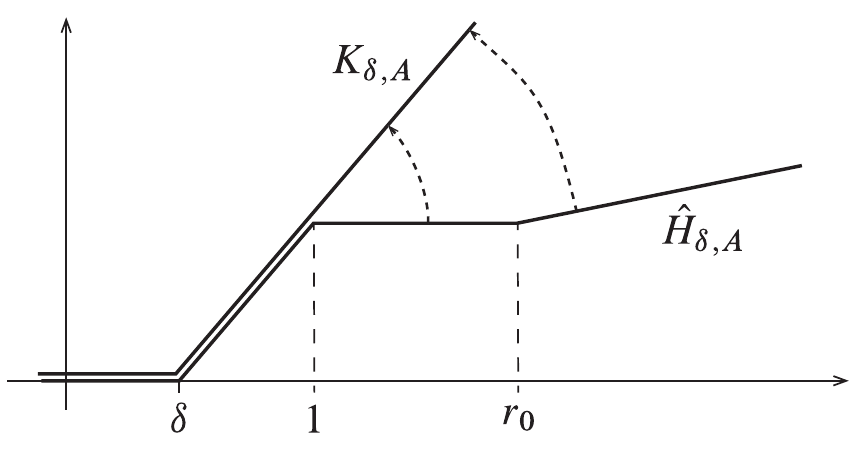}
\par\end{centering}
\caption{Homotopy from $\hat{H}_{\delta,A}$ to $K_{\delta,A}$.}

\label{Fig:homotopy_lift}

\end{figure}
The new orbits created by $\Phi_{F_{\bullet}}$ near $\ens{1}\times\del D$
will have action in the interval
\[
[T_{0}+A(\delta-1),\delta A-\eta_{A}]
\]
which, since $\varepsilon<\eta_{A}$ by assumption, is disjoint from
$(\delta A-\varepsilon,\infty)$. Hence these new orbits will all
appear out of the action window under consideration. Thus, Lemma \ref{lem:Homotopy_window}
assures us that $[\Phi_{2}]$ is an isomorphism.

STEP 3. Recall from Equation \ref{eq:j}, that we have a natural map

\[
\begin{tikzcd}
j_{K_\delta,A}:\Hf^\ast_{(\delta A-\varepsilon,+\infty)}(K_{\delta,A})
	\arrow{r}
&
\SH^\ast_{(\delta A-\varepsilon,+\infty)}(D)
\cong \SH^\ast(D) 
\end{tikzcd}.
\]The isomorphism to $\SH^{\ast}(D)$ follows from the fact that, by
construction, $\delta A-\varepsilon<0$. 

We define $\Psi_{\text{{I,II}}}:\HH^{\ast}(\CO_{\text{I,II}}^{\ast})\to\SH_{(\delta A-\varepsilon,+\infty)}^{\ast}(D)$
to the composition 
\begin{align*}
\Psi_{\text{I,II}}=j_{K_{\delta,A}}\comp[\Phi_{2}]\comp[\Phi_{1}].
\end{align*}

The morphism $\Psi$ is built in a similar fashion. We define it as
the composition of the maps 

\[
\begin{tikzcd}
\Hf^\ast(H_{\delta, A})
	\arrow{r}{\cong}[swap]{[\Phi_1']}
&
\Hf^\ast(\hat{H}_{\delta,A})
	\arrow{d}{[\Phi_2']}
\\
&
\Hf^\ast(K_{\delta,A})
	\arrow{r}{j_{K_{\delta,A}}}
&
\SH^\ast(D)
\cong 
\SH^\ast(D) 
\end{tikzcd}
\]Here, the isomorphism $[\Phi_{1}']$ follows from the fact that both
$H_{\delta,A}$ and $\hat{H}_{\delta,A}$ have the same slope at infinity.
We defined $[\Phi_{2}']$ to be the composition of the continuation
map $[\Phi^{K_{\delta,A}\hat{H}_{\delta,A}}]$ and the projection
$[\pi_{-\infty,\delta A-\varepsilon}^{+\infty+\infty}]:\Hf^{\ast}(K_{\delta,A})\to\Hf_{(\delta A-\varepsilon+\infty)}^{\ast}(K_{\delta,A})$
which is an isomorphism. The last map is given, just as in STEP 3,
by $j_{K_{\delta,A}}:\Hf^{\ast}(K_{\delta,A})\to\SH^{\ast}(D)$ .
By construction, we therefore have
\begin{align*}
\Psi=j_{K_{\delta,A}}\comp[\Phi_{2}']\comp[\Phi_{1}']=j_{K_{\delta,A}}\comp[\pi_{-\infty,\delta A-\varepsilon}^{+\infty,+\infty}]\comp[\Phi_{K_{\delta,A}\hat{H}_{\delta,A}}]\comp[\Phi_{1}']=j_{H_{\delta,A}}
\end{align*}
as desired.

Now, we need to prove that Diagram \eqref{triangle} commutes. Writing
the maps $\Psi$ and $\Psi_{\text{I,II}}$ explicitly, we have the
following diagram:

\begin{equation}
\label{square}
\begin{tikzcd}[row sep=40pt, column sep=60pt]
\Hf^\ast(H_{\delta,A})
	\arrow{r}{[\pi_{-\infty,A-\varepsilon}^{+\infty,+\infty}]}
	\arrow{d}{[\Phi_1']}
&
\HH^\ast(\CO_{\text{I,II}}^\ast)
	\arrow{d}[swap]{[\Phi_1]}
\\
\Hf^\ast(\hat{H}_{\delta,A})
	\arrow{d}{[\Phi_{K_{\delta,A}\hat{H}_{\delta,A}}]}
	\arrow{r}{[\pi_{-\infty,\delta A-\varepsilon}^{+\infty,+\infty}]}
&
\Hf^\ast_{(\delta A-\varepsilon,+\infty)}(\hat{H}_{\delta,A})
	\arrow{d}[swap]{[\Phi_2]}
\\
\Hf^\ast(K_{\delta, A})
\arrow{r}{[\pi_{-\infty,\delta A-\varepsilon}^{+\infty,+\infty}]}
&
\Hf^\ast_{(\delta A-\varepsilon,+\infty)}(K_{\delta,A})
	\arrow{d}[swap]{j_{K_{\delta,A}}}
\\
&
\SH^\ast(D) 
\end{tikzcd}
\end{equation}The top square in Diagram \ref{square} commutes because, since $\hat{H}_{\delta,A}\geq H_{\delta,A}$,
there exists a continuation map from $\Hf^{\ast}(H_{\delta,A})$ to
$\Hf_{(\delta A-\varepsilon,+\infty)}^{\ast}(\hat{H}_{\delta,A})$.
Now, since the projection $[\pi_{-\infty,\delta A-\varepsilon}^{+\infty,+\infty}]$
commutes with continuation maps (see Diagram \ref{cont_proj_res_comm}),
the bottom square in Diagram \ref{square} also commutes. Therefore,
we can conclude that Diagram \eqref{triangle} commutes. 

\subsubsection{Spectral invariant and spectral norm of $H_{\delta,A}$}

Recall that, by definition, 
\[
c(1,H_{\delta,A})=\inf\{c\in\RR\mid[\pi_{-\infty,\ell}^{+\infty,+\infty}]\comp[\iota_{-\infty,-\infty}^{\ell,+\infty}](1)=0\}
\]
Since $\Psi$ is a morphism of unital algebras, the commutative diagram
\eqref{triangle} assures us that $[\pi_{-\infty,A-\varepsilon}^{+\infty,+\infty}](1_{H_{\delta,A}})\neq0.$
Thus, from the exact triangle in cohomology induced by $[\iota_{-\infty,-\infty}^{A-\varepsilon,+\infty}]$
and $[\pi_{-\infty,A-\varepsilon}^{+\infty,+\infty}]$, we have $1\notin\im[\iota_{-\infty,-\infty}^{A-\varepsilon,+\infty}]$
and therefore,
\[
c(1,H_{\delta,A})\geq A-\varepsilon.
\]

Now, we turn our attention to the spectral norm $\gamma(H_{\delta,A})$.
We know from Lemma \ref{lem:C} that $c(1,H_{\delta,A}),c(1,\overline{{H}}_{\delta,A})\geq0$.
It thus follows from the previous inequality that 
\[
\gamma(H_{\delta,A})=c(1,H_{\delta,A})+c(1,\overline{{H}}_{\delta,A})\geq A-\varepsilon
\]
as desired. This completes the proof.

\subsection{Proof of Lemma \ref{lem:C} }

We give a proof of Lemma \ref{lem:C} which relies on the decomposition
of the Floer complex induced by the barricade. We expect that Lemma
\ref{lem:C} could also be proven using Poincaré duality and Lemma
4.1 of \cite{GaTa20}.

Let $H\in\mathcal{H}_{r_{0}}$ with slope $0<\tau_{H}<T_{0}$. Consider
a linear homotopy $F_{\bullet}$ from $F_{+}=K_{r_{0},\tau}$ for
$0<\tau<T_{0}$(see Figure \ref{Fig:K_Hamiltonian}) to $F_{-}=H$
. There exists a small perturbation $f_{\bullet}$ of $F_{\bullet}$
and an almost complex structure $J$ such that the pairs $(f_{\bullet},J)$
and $(f_{\pm},J)$ admit a barricade on $B_{r_{0},\varepsilon}$ for
$\varepsilon>0$ small enough. Fix $\delta>0$. The construction allows
us to choose $J$ time independent and $f$ such that
\begin{align*}
-\delta\leq\int_{0}^{1}\min_{x\in\hat{D}\setminus(r_{0},+\infty)\times\partial D}(f_{-}-H)\dd t\leq\delta.
\end{align*}
We may assume further that $f_{+}$ has a local minimum point $p\in\Dc$,
since $f_{+}$ is $C^{2}$-small there. It follows from Lemma \ref{lem:Standard_cohomology}
that $1_{f_{+}}=[p]\in\Hf^{\ast}(f_{+})$ is the image of the unit
$e_{D}\in\HH^{\ast}(D)$ under the isomorphism $\Phi_{f_{+}}:\HH^{\ast}(D)\to\Hf^{\ast}(f_{+})$.
Moreover, Lemma \ref{lem:Same_Slope} assures us that the isomorphism
$[\Phi_{f_{\bullet}}]:\Hf^{\ast}(f_{+})\to\Hf^{\ast}(f_{-})$ induced
by the continuation morphism $\Phi_{f_{\bullet}}:\CF^{\ast}(f_{+})\to\CF^{\ast}(f_{-})$
preserves the unit. To summarize, we have 
\begin{align*}
\Phi_{f_{+}}(e_{D})=[p]=1_{f_{+}}\quad\text{{and}\ensuremath{\quad}}[\Phi_{f_{\bullet}}(p)]=[\Phi_{f_{\bullet}}](1_{f_{+}})=1_{f_{-}}.
\end{align*}

By the continuity of spectral invariants, we know that 
\begin{align*}
c(1,H)-c(1,f_{-})\geq\int_{0}^{1}\min_{x\in\hat{D}\setminus(r_{0},+\infty)\times\partial D}(f_{-}-H)\dd t & .
\end{align*}
Therefore, by our choice of $f_{-}$, we have $c(1,H)\geq-\delta+c(1,f_{-}).$
To complete the proof, it suffices to show that $c(1,f_{-})\geq-k\delta$
for $k$\textgreater 0 independent of $f_{-}$. However, the definition
of spectral invariants guarantees the existence of $q\in\CF^{\ast}(f_{-})$
cohomologous to $1$ for which $c(1,f_{-})\geq\mathcal{A}_{f_{-}}(q)-\delta$.
We thus only need to prove that $\mathcal{A}_{f_{-}}(q)\geq-\delta.$

Recall that, by the barricade construction, $\CO^{\ast}(\Db,f_{-})$
forms a sub-complex of $\CF^{\ast}(f_{-})$. Moreover, 
\begin{align*}
\CO^{\ast}(\Dc,f_{-})=\frac{\CF^{\ast}(f_{-})}{\CO^{\ast}(\Db,f_{-})}
\end{align*}
where we denote the projection by $\pi_{-}^{\text{c}}:\CF^{\ast}(f_{-})\to\CO^{\ast}(\Dc,f_{-})$.
This allows us to write 
\begin{align*}
\Phi_{f_{\bullet}}(p) & =p_{\text{b}}+p_{\text{c}}\ \text{and}\:q=p_{\text{b}}+p_{\text{c}}+\partial(r_{\text{b}}+r_{\text{c}})
\end{align*}
for $p_{\text{b}},r_{\text{b}}\in\CO^{\ast}(\Db,f_{-})$ and $p_{\text{c}},r_{\text{c}}\in\CO^{\ast}(\Dc,f_{-})$.
Again, from the barricade construction, we have 
\begin{align*}
\partial(r_{\text{b}}+r_{0})=r_{\text{bb}}+r_{\text{cb}}+r_{\text{cc}}
\end{align*}
where $r_{\text{bb}},r_{\text{cb}}\in\CO^{\ast}(\Db,f_{-})$ and $r_{\text{cc}}\in\CO^{\ast}(\Dc,f_{-})$.
Notice that since $f_{-}$ is $C^{2}$-small on $\Dc$, $\mathcal{A}_{f_{-}}(p_{\text{c}}+r_{\text{cc}})\geq-\delta.$
Thus, if $p_{\text{c}}+r_{\text{cc}}\neq0$, we have 
\begin{align*}
\mathcal{A}_{f_{-}}(q)=\mathcal{A}_{f_{-}}(p_{\text{b}}+p_{\text{c}}+r_{\text{bb}}+r_{\text{cb}}+r_{\text{cc}})\geq\mathcal{A}_{f_{-}}(p_{\text{c}}+r_{\text{cc}})\geq-\delta.
\end{align*}

We now prove that $[p_{\text{c}}]\in\HH^{\ast}(\Dc,f_{-})$ is not
zero which is equivalent to showing that $p_{\text{c}}+r_{\text{cc}}\neq0$.
Denote by $\Phi_{\bar{f}_{\bullet}}:\CF^{\ast}(f_{-})\to\CF^{\ast}(f_{+})$
the\emph{ }continuation map generated by the inverse homotopy $\overline{f}_{s}=f_{-s}$.
We know that both $\Phi_{\bar{f}_{\bullet}}\comp\Phi_{f_{\bullet}}$
and $\Phi_{f_{\bullet}}\comp\Phi_{\bar{f}_{\bullet}}$ are chain homotopic
to the identity :
\begin{align*}
\Phi_{\bar{f}_{\bullet}}\comp\Phi_{f_{\bullet}}-\mathrm{id}_{+} & =\partial_{+}\comp\Psi_{+}-\Psi_{+}\comp\partial_{+}\\
\Phi_{f_{\bullet}}\comp\Phi_{\bar{f}_{\bullet}}-\mathrm{id}_{-} & =\partial_{-}\comp\Psi_{-}-\Psi_{-}\comp\partial_{-}
\end{align*}
for the differentials $\partial_{\pm}:\CF^{\ast}(f_{\pm})\to\CF^{\ast+1}(f_{\pm})$
and chain homotopies $\Psi_{\pm}:\CF^{\ast}(f_{\pm})\to\CF^{\ast-1}(f_{\pm})$.
(In fact, for our purpose here, we only need the first homotopy relation.)
Since $\Psi_{\pm}$ also obey the rules of the barricade by Lemma
\ref{lem:barricade_homotopy2}, the composition of the projections
$\Phic_{f_{\bullet}}:\CO^{\ast}(\Dc,f_{+})\to\CO^{\ast}(\Dc,f_{-})$
and $\Phic_{\bar{f}_{\bullet}}:\CO^{\ast}(\Dc,f_{-})\to\CO^{\ast}(\Dc,f_{+})$
are chain homotopic to the identity on $\CO^{\ast}(\Dc,f_{+})$ by
Lemma \ref{lem:chain_homotopy}. Therefore, on cohomology, the morphism
\begin{align*}
[\Phic_{\bar{f}_{\bullet}}\comp\Phi_{f_{\bullet}}^{c}] & :\HH^{\ast}(\Dc,f_{+})\to\HH^{\ast}(\Dc,f_{+})
\end{align*}
is given by the identity and since $[p]\neq0$,
\begin{align*}
[p_{\text{c}}]=[\Phic_{f_{\bullet}}(p)]=[\Phic_{f_{\bullet}}]([p])\neq0.
\end{align*}
This concludes the proof.

\subsection{Proof of Lemma \ref{lem:skeleton}}

\noindent Let $0<\delta<1$ be small enough so that 
\begin{align*}
\delta A<\delta A+\delta\eta_{A}<\eta_{A}.
\end{align*}
Then, following the proof of Theorem \ref{thm:finite_diameter}, we
have that 
\begin{align*}
c(1,H_{\delta,A})\geq A-\delta(A+\eta_{A}).
\end{align*}
Notice that $H_{\delta,A}$ converges uniformly as $\delta\to0$ to
the continuous function $H_{0,A}$ (see Figure \ref{Fig:H0A}). Then,
by continuity of spectral invariants and the previous equation, we
have

\begin{align*}
c(1,H_{0,A})=\lim_{\delta\to0}c(1,H_{\delta,A})\geq\lim_{\delta\to0}\left(A-\delta(A+\eta_{A})\right)=A.
\end{align*}

\begin{figure}[h]
\begin{centering}
\includegraphics{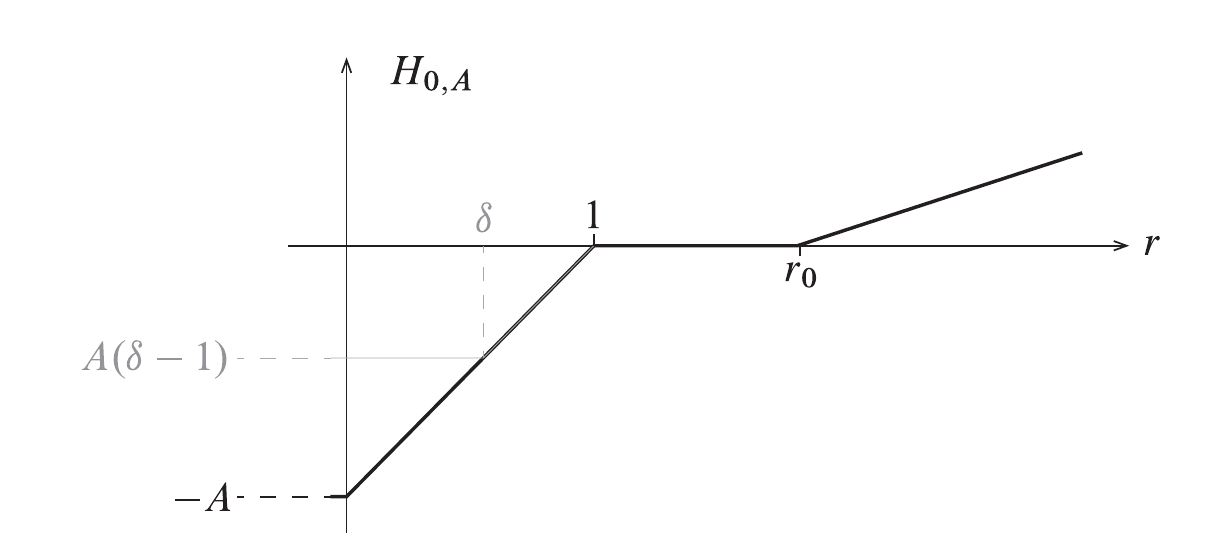}
\par\end{centering}
\caption{The continuous Hamiltonian $H_{0,A}$.}

\label{Fig:H0A}
\end{figure}
\noindent Moreover, since $H_{0,A}\geq-A$, continuity of spectral
invariants yields 
\begin{align*}
c(1,H_{0,A})\leq\max_{x\in D}-H_{0,A}=A
\end{align*}
which allows us to conclude that $c(1,H_{0,A})=A.$

First, we prove the Lemma for Hamiltonians which are constant on an
open neighborhood of the Skeleton of $D$. Consider an autonomous
Hamiltonian $H\in\mathcal{C}(D)$ such that $H\big|_{V}=-A$ and $-A\leq H\leq0$
for an open neighborhood $V$ of $\Sk(D)$ and a constant $A>0$.
The last condition on $H$ allows us to use continuity of spectral
invariance to conclude that 
\begin{align}
c(1,H)\leq A.\label{eq:c<A}
\end{align}
All we need to do now is prove that $A$ bounds $c(1,H)$ from below. 

Define $F\in\mathcal{C}(D)$ to be the continuous autonomous Hamiltonian
that agrees with $H_{0,A/r'}$ on $D$ for some $0<r'<1$. Since $H\big|_{V}=-A$
, we can choose $r'$ so that the $r'$-contraction $F_{r'}$ of $F$
under the Liouville flow (see Equation \ref{eq:Hr} and Figure \ref{Fig:HFFr}),
has support in $V$ and $-A\leq F_{r'}\leq0$. Therefore,
\begin{align}
F_{r'}(x)\geq H(x),\quad\forall x\in D.\label{eq:FrH}
\end{align}

\begin{figure}[h]
\begin{centering}
\includegraphics{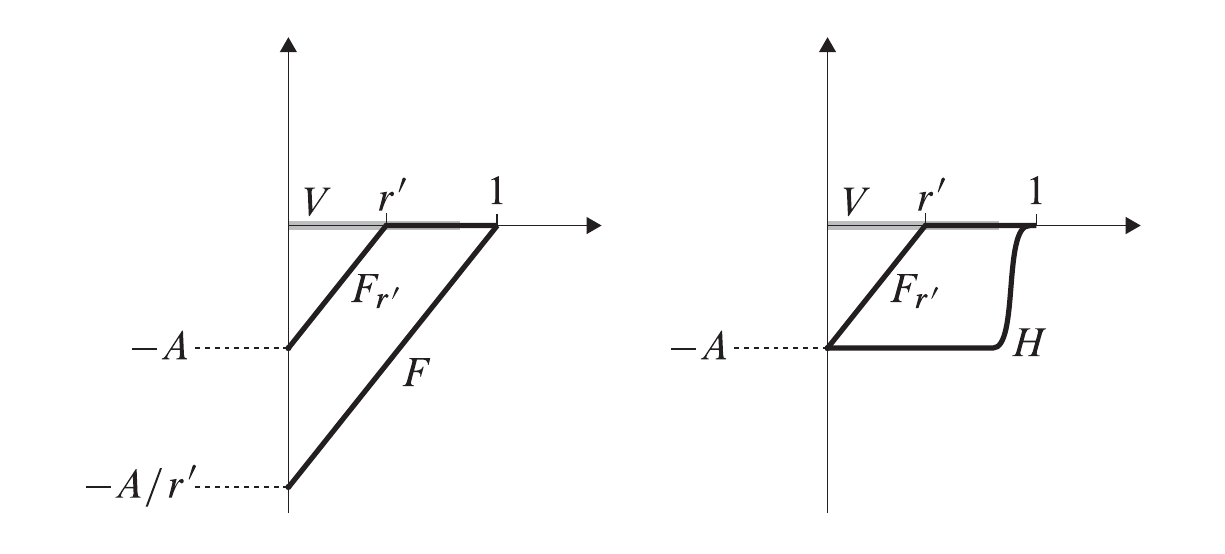}
\par\end{centering}
\caption{The Hamiltonians $F$, $F_{r'}$ and $H$.}

\label{Fig:HFFr}
\end{figure}

From the contraction principle stated in Lemma \ref{lem:Contraction_principle}
and the computation of $c(1,H_{0,A})$ above, we have 
\begin{align*}
c(1,F_{r'})=r'c(1,F)=r'c(1,H_{0,A/r})=A.
\end{align*}
This computation and Equation \ref{eq:FrH} yield, by virtue of the
monoticity of spectral invariants, the lower bound $A=c(1,F_{r'})\leq c(1,H)$
as desired. In conjunction with Equation \ref{eq:c<A}, we conclude
that $c(1,H)=A$. 

Now, we prove the Lemma in general. Suppose $H\big|_{\Sk(D)}=-A$
and $-A\leq H\leq0$. For any $\varepsilon\in(0,1)$, there exists
a compactly supported Hamiltonian $H_{\varepsilon}$ such that $H_{\varepsilon}\big|_{V_{\varepsilon}}=-A$
for an open neighborhood $V_{\varepsilon}$ of $\Sk(D)$ and $H_{\varepsilon}\leq H$
everywhere. Indeed, define $H_{\varepsilon}$ as follows : $H_{\varepsilon}\big|_{\Sk(D)}=-A$,
\begin{align*}
H_{\varepsilon}\big|_{D^{\varepsilon}\setminus\Sk(D)} & =\beta_{\varepsilon}(r)H+(1-\beta_{\varepsilon}(r))(-A)
\end{align*}
where $\beta_{\varepsilon}:(0,1)\to\RR$ is such that
\begin{itemize}
\item $\beta_{\varepsilon}\big|_{(0,\varepsilon]}\equiv0$,
\item $\beta'_{\varepsilon}\big|_{(\varepsilon,2\varepsilon/3)}>0,$
\item $\beta_{\varepsilon}\big|_{(2\varepsilon/3,1)}\equiv1.$
\end{itemize}
Then, $H_{\varepsilon}$ satisfies the required conditions and converges
uniformly to $H$ as $\varepsilon\to0$. We have $c(1,H_{\varepsilon})$
by the previous computation and by continuity of spectral invariants,
we can conclude that 
\begin{align*}
c(1,H)=c(1,H_{\varepsilon})=A.
\end{align*}
This completes the proof. 

\subsection{Proof of Theorem \ref{thm:B2}}

\noindent Let $H\in\mathcal{C}(D)$ be an autonomous Hamiltonian
such that $H\big|_{V}=-1$ and $-1\leq H\leq0$ everywhere for an
open neighborhood $V$ of $\Sk(D)$. 

Define $\iota:\RR\to\Ham_{c}(D)$ as 
\begin{align*}
\iota(s)=\varphi_{sH} & ,
\end{align*}
where $\varphi_{sH}\in\Ham_{c}(D)$ is the time-one map associated
to $sH$. We claim that $\iota$ is the desired embedding.

We first bound $d_{\gamma}(\iota(s),\iota(s'))$ from above. If $F\in\mathcal{C}(D)$,
then $\gamma(\varphi_{F})\leq\norm{F}$. Moreover, since $H$ is autonomous,
$sH\#\overline{{s'H}}=(s-s')H$. Therefore, 
\begin{align*}
d_{\gamma}(\iota(s),\iota(s')) & =\gamma(\iota(s)\iota(s')^{-1})\leq\norm{(s-s')H}=\abs{s-s'}.
\end{align*}

Now, we bound $d_{\gamma}(\iota(s),\iota(s'))$ from below. Since
$d_{\gamma}$ is symmetric, we can assume that $s\geq s'$. Then,
by Lemma \ref{lem:c geq 0} and Lemma \ref{lem:skeleton}, we have
\begin{align*}
d_{\gamma}(\iota(s),\iota(s'))\geq c(1,(s-s')H)=s-s' & ,
\end{align*}
which completes the proof. 

\bibliographystyle{alpha}
\bibliography{references.bib}

\end{document}